\documentclass[11pt]{article}

\usepackage[utf8]{inputenc}
\usepackage[dvips,letterpaper,margin=1in]{geometry}
\usepackage{sty}
\usepackage{color}

\definecolor{AfonsoBlue}{RGB}{30,65,123}

\title{Notes on Computational Hardness of Hypothesis Testing: Predictions using the Low-Degree Likelihood Ratio}

\makeatletter
\renewcommand*{\@fnsymbol}[1]{\ensuremath{\ifcase#1\or *\or \ddagger\or
    \mathsection\or \mathparagraph\or \|\or **\or \dagger\dagger
    \or \ddagger\ddagger \else\@ctrerr\fi}}
\makeatother

\usepackage{authblk}
\author[1]{Dmitriy Kunisky\thanks{Email: \textit{kunisky@cims.nyu.edu}. Partially supported by NSF grants DMS-1712730 and DMS-1719545.}}
\author[1]{Alexander S.\ Wein\thanks{Email: \textit{awein@cims.nyu.edu}. Partially supported by NSF grant DMS-1712730 and by the Simons Collaboration on Algorithms and Geometry.}}
\author[1,2]{Afonso S.\ Bandeira\thanks{Email: \textit{bandeira@cims.nyu.edu}. Partially supported by NSF grants DMS-1712730 and DMS-1719545, and by a grant from the Sloan
Foundation.}}

\affil[1]{Department of Mathematics, Courant Institute of Mathematical Sciences, New York University}
\affil[2]{Center for Data Science, New York University}

\date{July 26, 2019}

\begin{document}

\maketitle

\begin{abstract}

These notes survey and explore an emerging method, which we call the \emph{low-degree method}, for predicting and understanding statistical-versus-computational tradeoffs in high-dimensional inference problems. In short, the method posits that a certain quantity -- the second moment of the \emph{low-degree likelihood ratio} -- gives insight into how much computational time is required to solve a given hypothesis testing problem, which can in turn be used to predict the computational hardness of a variety of statistical inference tasks. While this method originated in the study of the sum-of-squares (SoS) hierarchy of convex programs, we present a self-contained introduction that does not require knowledge of SoS. In addition to showing how to carry out predictions using the method, we include a discussion investigating both rigorous and conjectural consequences of these predictions.

These notes include some new results, simplified proofs, and refined conjectures. For instance, we point out a formal connection between spectral methods and the low-degree likelihood ratio, and we give a sharp low-degree lower bound against subexponential-time algorithms for tensor PCA.
\end{abstract}

\newpage

\setcounter{tocdepth}{2}
\tableofcontents

\newpage

\section*{Overview}
\addcontentsline{toc}{section}{Overview}

Many problems in high-dimensional statistics are believed to exhibit gaps between what can be achieved \emph{information-theoretically} (or \emph{statistically}, i.e., with unbounded computational power) and what is possible with bounded computational power (e.g., in polynomial time). Examples include finding planted cliques \cite{J-clique,DM-clique,MPW-clique,pcal} or dense communities \cite{block-model-1,block-model-2,HS-bayesian} in random graphs, extracting variously structured principal components of random matrices \cite{BR-sparse,LKZ-mmse,LKZ-sparse} or tensors \cite{HSS-tensor,sos-hidden}, and solving or refuting random constraint satisfaction problems \cite{alg-barriers,refuting-any-csp}.

Although current techniques cannot prove that such average-case problems require super-polynomial time (even assuming $P \ne NP$), various forms of rigorous evidence for hardness have been proposed. These include:
\begin{itemize}
    \item failure of Markov chain Monte Carlo methods \cite{J-clique,DFJ-mcmc};
    \item failure of local algorithms \cite{GS-local,DM-hidden,tensor-local,subopt-local-maxcut};
    \item methods from statistical physics which suggest failure of belief propagation or approximate message passing algorithms \cite{block-model-1,block-model-2,LKZ-mmse, LKZ-sparse} (see \cite{stat-phys-survey} for a survey or \cite{BPW-phys-notes} for expository notes);
    \item structural properties of the solution space \cite{alg-barriers,KMRSZ,GS-local,GZ-reg,GZ-clique};
    \item geometric analysis of non-convex optimization landscapes \cite{ABC,matrix-tensor};
    \item reductions from planted clique (which has become a ``canonical'' problem believed to be hard in the average case) \cite{BR-sparse,HWX-reduction,WBS,hard-rip,bresler-sparse,bresler-pca};
    \item lower bounds in the statistical query model \cite{sq-kearns,sq-half,sq-clique,sq-sat,sq-gaussian,sq-robust};
    \item lower bounds against the sum-of-squares hierarchy \cite{grig-parity,sch-parity,DM-clique,MPW-clique,HSS-tensor,MW-sos,pcal,sos-hidden} (see \cite{sos-survey} for a survey).
\end{itemize} 

\noindent In these notes, we survey another emerging method, which we call the \emph{low-degree method}, for understanding computational hardness in average-case problems. 
In short, we explore a conjecture that the behavior of a certain quantity -- the second moment of the \emph{low-degree likelihood ratio} -- reveals the computational complexity of a given statistical task. 
We find the low-degree method particularly appealing because it is simple, widely applicable, and can be used to study a wide range of time complexities (e.g., polynomial, quasipolynomial, or nearly-exponential).
Furthermore, rather than simply positing a certain ``optimal algorithm,'' the underlying conjecture captures an interpretable structural feature that seems to dictate whether a problem is easy or hard. 
Finally, and perhaps most importantly, predictions using the low-degree method have been carried out for a variety of average-case problems, and so far have always reproduced widely-believed results.

Historically, the low-degree method arose from the study of the sum-of-squares (SoS) semidefinite programming hierarchy. In particular, the method is implicit in the \emph{pseudo-calibration} approach to proving SoS lower bounds \cite{pcal}. Two concurrent papers \cite{HS-bayesian,sos-hidden} later articulated the idea more explicitly. In particular, Hopkins and Steurer \cite{HS-bayesian} were the first to demonstrate that the method can capture sharp thresholds of computational feasibility such as the Kesten--Stigum threshold for community detection in the stochastic block model. The low-degree method was developed further in the PhD thesis of Hopkins \cite{sam-thesis}, which includes a precise conjecture about the complexity-theoretic implications of low-degree predictions. In comparison to sum-of-squares lower bounds, the low-degree method is much simpler to carry out and appears to always yield the same results for natural average-case problems.

In these notes, we aim to provide a self-contained introduction to the low-degree method; we largely avoid reference to SoS and instead motivate the method in other ways. We will briefly discuss the connection to SoS in Section~\ref{sec:sos}, but we refer the reader to \cite{sam-thesis} for an in-depth exposition of these connections.

These notes are organized as follows. In Section~\ref{sec:decisiontheory}, we present the low-degree method and motivate it as a computationally-bounded analogue of classical statistical decision theory. In Section~\ref{sec:agn}, we show how to carry out the low-degree method for a general class of additive Gaussian noise models. In Section~\ref{sec:examples}, we specialize this analysis to two classical problems: the spiked Wigner matrix and spiked Gaussian tensor models. Finally, in Section~\ref{sec:ldlr-conj-2}, we discuss various forms of heuristic and formal evidence for correctness of the low-degree method; in particular, we highlight a formal connection between low-degree lower bounds and the failure of spectral methods (Theorem~\ref{thm:spectral-hard}).

\section{Towards a Computationally-Bounded Decision Theory}\label{sec:decisiontheory}

\subsection{Statistical-to-Computational Gaps in Hypothesis Testing}

The field of \emph{statistical decision theory} (see, e.g., \cite{LR-sdt,LeCam-sdt} for general references) is concerned with the question of how to decide optimally (in some quantitative sense) between several statistical conclusions.
The simplest example, and the one we will mainly be concerned with here, is that of \emph{simple hypothesis testing}: we observe a dataset that we believe was drawn from one of two probability distributions, and want to make an inference (by performing a statistical \emph{test}) about which distribution we think the dataset was drawn from.

However, one important practical aspect of statistical testing usually is not included in this framework, namely the \emph{computational cost} of actually performing a statistical test.
In these notes, we will explore ideas from a line of recent research about how one mathematical method of classical decision theory might be adapted to predict the capabilities and limitations of \emph{computationally bounded statistical tests}.

The basic problem that will motivate us is the following.
Suppose $\PPP = (\PP_n)_{n \in \NN}$ and $\QQQ = (\QQ_n)_{n \in \NN}$ are two sequences of probability distributions over a common sequence of measurable spaces $\sS = ((\sS_n, \sF_n))_{n \in \NN}$.
(In statistical parlance, we will think throughout of $\PPP$ as the model of the \emph{alternative hypothesis} and $\QQQ$ as the model of the \emph{null hypothesis}. Later on, we will consider hypothesis testing problems where the distributions $\PPP$ include a ``planted'' structure, making the notation a helpful mnemonic.)
Suppose we observe $\bY \in \sS_n$ which is drawn from one of $\PP_n$ or $\QQ_n$.
We hope to recover this choice of distribution in the following sense.

\begin{definition}
    We say that a sequence of events $(A_n)_{n \in \NN}$ with $A_n \in \sF_n$ occurs with \emph{high probability (in $n$)} if the probability of $A_n$ tends to 1 as $n \to \infty$.
\end{definition}

\begin{definition}\label{def:stat-ind}
    A sequence of (measurable) functions $f_n: \sS_n \to \{ \tp, \tq \}$ is said to \emph{strongly distinguish}\footnote{We will only consider this so-called \emph{strong} version of distinguishability, where the probability of success must tend to 1 as $n \to \infty$, as opposed to the \emph{weak} version where this probability need only be bounded above $\frac{1}{2}$. For high-dimensional problems, the strong version typically coincides with important notions of estimating the planted signal (see Section~\ref{sec:extensions}), whereas the weak version is often trivial.} $\PPP$ and $\QQQ$ if $f_n(\bY) = \tp$ with high probability when $\bY \sim \PP_n$, and $f_n(\bY) = \tq$ with high probability when $\bY \sim \QQ_n$. If such $f_n$ exist, we say that $\PPP$ and $\QQQ$ are \emph{statistically distinguishable}.
\end{definition}

In our computationally bounded analogue of this definition, let us for now only consider polynomial time tests (we will later consider various other restrictions on the time complexity of $f_n$, such as subexponential time).
Then, the analogue of Definition~\ref{def:stat-ind} is the following.

\begin{definition}
    $\PPP$ and $\QQQ$ are said to be \emph{computationally distinguishable} if there exists a sequence of measurable \textbf{and computable in time polynomial in $\bm n$} functions $f_n: \sS_n \to \{\tp, \tq\}$ such that $f_n$ strongly distinguishes $\PPP$ and $\QQQ$.
\end{definition}

\noindent Clearly, computational distinguishability implies statistical distinguishability.
On the other hand, a multitude of theoretical evidence suggests that statistical distinguishability does not in general imply computational distinguishability.
Occurrences of this phenomenon are called \emph{statistical-to-computational (stat-comp) gaps}.
Typically, such a gap arises in the following slightly more specific way. Suppose the sequence $\PPP$ has a further dependence on a \emph{signal-to-noise} parameter $\lambda > 0$, so that $\PPP_\lambda = (\PP_{\lambda, n})_{n \in \NN}$.
This parameter should describe, in some sense, the strength of the structure present under $\PPP$ (or, in some cases, the number of samples received). The following is one canonical example.

\begin{example}[Planted Clique Problem \cite{J-clique,kucera-clique}]
Under the null model $\QQ_n$, we observe an $n$-vertex Erd\H{o}s-R\'enyi graph $\sG(n,1/2)$, i.e., each pair $\{i, j\}$ of vertices is connected with an edge independently with probability $1/2$. The signal-to-noise parameter $\lambda$ is an integer $1 \le \lambda \le n$. Under the planted model $\PP_{\lambda,n}$, we first choose a random subset of vertices $S \subseteq [n]$ of size $|S| = \lambda$ uniformly at random. We then observe a graph where each pair $\{i, j\}$ of vertices is connected with probability $1$ if $\{i,j\} \subseteq S$ and with probability $1/2$ otherwise. In other words, the planted model consists of the union of $\sG(n,1/2)$ with a planted \emph{clique} (a fully-connected subgraph) on $\lambda$ vertices.
\end{example}

\noindent As $\lambda$ varies, the problem of testing between $\PPP_\lambda$ and $\QQQ$ can change from statistically impossible, to statistically possible but computationally hard, to computationally easy.
That is, there exists a threshold $\lambda_{\mathsf{stat}}$ such that for any $\lambda > \lambda_{\mathsf{stat}}$, $\PPP_{\lambda}$ and $\QQQ$ are statistically distinguishable, but for $\lambda < \lambda_{\mathsf{stat}}$ are not.
There also exists a threshold $\lambda_{\mathsf{comp}}$ such that for any $\lambda > \lambda_{\mathsf{comp}}$, $\PPP_\lambda$ and $\QQQ$ are computationally distinguishable, and (conjecturally) for $\lambda < \lambda_{\mathsf{comp}}$ are not.
Clearly we must have $\lambda_{\mathsf{comp}} \geq \lambda_{\mathsf{stat}}$, and a stat-comp gap corresponds to strict inequality $\lambda_{\mathsf{comp}} > \lambda_{\mathsf{stat}}$. For instance, the two models in the planted clique problem are statistically distinguishable when $\lambda \ge (2+\varepsilon) \log_2 n$ (since $2\log_2 n$ is the typical size of the largest clique in $\sG(n,1/2)$), so $\lambda_{\mathsf{stat}} = 2 \log_2 n$. However, the best known polynomial-time distinguishing algorithms only succeed when $\lambda = \Omega(\sqrt{n})$ \cite{kucera-clique,AKS-clique}, and so (conjecturally) $\lambda_{\mathsf{comp}} \approx \sqrt{n}$, a large stat-comp gap.

The remarkable method we discuss in these notes allows us, through a relatively straightforward calculation, to predict the threshold $\lambda_{\mathsf{comp}}$ for many of the known instances of stat-comp gaps.
We will present this method as a modification of a classical second moment method for studying $\lambda_{\mathsf{stat}}$.

\subsection{Classical Asymptotic Decision Theory}

In this section, we review some basic tools available from statistics for understanding statistical distinguishability.
We retain the same notations from the previous section in the later parts, but in the first part of the discussion will only be concerned with a single pair of distributions $\PP$ and $\QQ$ defined on a single measurable space $(\sS, \sF)$.
For the sake of simplicity, let us assume in either case that $\PP_n$ (or $\PP$) is absolutely continuous with respect to $\QQ_n$ (or $\QQ$, as appropriate).\footnote{For instance, what will be relevant in the examples we consider later, any pair of non-degenerate multivariate Gaussian distributions satisfy this assumption.}

\subsubsection{Basic Notions}

We first define the basic objects used to make hypothesis testing decisions, and some ways of measuring their quality.

\begin{definition}
    A \emph{test} is a measurable function $f: \sS \to \{\tp, \tq\}$.
\end{definition}

\begin{definition}
    The \emph{type I error} of $f$ is the event of falsely rejecting the null hypothesis, i.e., of having $f(\bY) = \tp$ when $\bY \sim \QQ$.
    The \emph{type II error} of $f$ is the event of falsely failing to reject the null hypothesis, i.e., of having $f(\bY) = \tq$ when $\bY \sim \PP$.
    The probabilities of these errors are denoted
    \begin{align*}
        \alpha(f) &\colonequals \QQ\left(f(\bY) = \tp\right), \\
        \beta(f) &\colonequals \PP\left(f(\bY) = \tq\right).
    \end{align*}
    The probability $1 - \beta(f)$ of correctly rejecting the null hypothesis is called the \emph{power} of $f$.
\end{definition}
\noindent
There is a tradeoff between type I and type II errors.
For instance, the trivial test that always outputs $\tp$ will have maximal power, but will also have maximal probability of type I error, and vice-versa for the trivial test that always outputs $\tq$.
Thus, typically one fixes a tolerance for one type of error, and then attempts to design a test that minimizes the probability of the other type.

\subsubsection{Likelihood Ratio Testing}

We next present the classical result showing that it is in fact possible to identify the test that is optimal in the sense of the above tradeoff.\footnote{It is important to note that, from the point of view of statistics, we are restricting our attention to the special case of deciding between two ``simple'' hypotheses, where each hypothesis consists of the dataset being drawn from a specific distribution. Optimal testing is more subtle for ``composite'' hypotheses in parametric families of probability distributions, a more typical setting in practice. The mathematical difficulties of this extended setting are discussed thoroughly in \cite{LR-sdt}.}

\begin{definition}
    Let $\PP$ be absolutely continuous with respect to $\QQ$. The \emph{likelihood ratio}\footnote{For readers not familiar with the Radon--Nikodym derivative: if $\PP$, $\QQ$ are discrete distributions then $L(\bY) = \PP(\bY)/\QQ(\bY)$; if $\PP$, $\QQ$ are continuous distributions with density functions $p$, $q$ (respectively) then $L(\bY) = p(\bY)/q(\bY)$.} of $\PP$ and $\QQ$ is
    \begin{equation*}
        L(\bY) \colonequals \frac{d\PP}{d\QQ}(\bY).
    \end{equation*}
    The \emph{thresholded likelihood ratio test} with threshold $\eta$ is the test
    \begin{equation*}
        L_{\eta}(\bY) \colonequals \left\{\begin{array}{lcl} \tp & : & L(\bY) > \eta \\ \tq & : & L(\bY) \leq \eta \end{array}\right\}.
    \end{equation*}
\end{definition}

\noindent Let us first present a heuristic argument for why thresholding the likelihood ratio might be a good idea. Specifically, we will show that the likelihood ratio is optimal in a particular ``$L^2$ sense'' (which will be of central importance later), i.e., when its quality is measured in terms of first and second moments of a testing quantity.

\begin{definition}\label{def:hilbert}
    For (measurable) functions $f,g: \sS \to \RR$, define the inner product and norm induced by $\QQ$:
    \begin{align*}
        \la f, g \ra &\colonequals \Ex_{\bY \sim \QQ}\left[ f(\bY) g(\bY) \right], \\
        \|f\| &\colonequals \sqrt{\langle f,f \rangle}.
    \end{align*}
    Let $L^2(\QQ)$ denote the Hilbert space consisting of functions $f$ for which $\|f\| < \infty$, endowed with the above inner product and norm.\footnote{For a more precise definition of $L^2(\QQ_n)$ (in particular including issues around functions differing on sets of measure zero) see a standard reference on real analysis such as \cite{SS-real}.}
\end{definition}

\begin{proposition}
    \label{prop:lr-optimal-l2}
    If $\PP$ is absolutely continuous with respect to $\QQ$, then the unique solution $f^*$ of the optimization problem
    \begin{equation*}
    \begin{array}{ll}
        \text{maximize} & \displaystyle \Ex_{\bY \sim \PP} [f(\bY)] \\[10pt]
        \text{subject to} & \displaystyle \Ex_{\bY \sim \QQ} [f(\bY)^2] = 1
    \end{array}
    \end{equation*}
    is the (normalized) likelihood ratio
    \[
        f^\star = L/\|L\|,
    \]
    and the value of the optimization problem is $\|L\|$.
\end{proposition}
\begin{proof}
    We may rewrite the objective as
    \[\Ex_{\bY \sim \PP} f(\bY) = \Ex_{\bY \sim \QQ} \left[L(\bY) f(\bY) \right] = \langle L,f \rangle,\]
    and rewrite the constraint as $\|f\| = 1$. The result now follows since $\langle L, f \rangle \leq \|L\| \cdot \|f\| = \|L\|$ by the Cauchy-Schwarz inequality, with equality if and only if $f$ is a scalar multiple of $L$.
\end{proof}
\noindent
In words, this means that if we want a function to be as large as possible in expectation under $\PP$ while remaining bounded (in the $L^2$ sense) under $\QQ$, we can do no better than the likelihood ratio.
We will soon return to this type of $L^2$ reasoning in order to devise computationally-bounded statistical tests.

The following classical result shows that the above heuristic is accurate, in that the thresholded likelihood ratio tests achieve the optimal tradeoff between type I and type II errors.
\begin{lemma}[Neyman--Pearson Lemma \cite{N-P}]
    \label{lem:neyman-pearson}
    Fix an arbitrary threshold $\eta \ge 0$. Among all tests $f$ with $\alpha(f) \leq \alpha(L_\eta) = \QQ(L(\bY) > \eta)$, $L_{\eta}$ is the test that maximizes the power $1 - \beta(f)$.
\end{lemma}
\noindent 
We provide the standard proof of this result in Appendix~\ref{app:neyman-pearson} for completeness. (The proof is straightforward but not important for understanding the rest of these notes, and it can be skipped on a first reading.)

\subsubsection{Le Cam's Contiguity}
\label{sec:contig}

Since the likelihood ratio is, in the sense of the Neyman--Pearson lemma, an optimal statistical test, it stands to reason that it should be possible to argue about statistical distinguishability solely by computing with the likelihood ratio.
We present one simple method by which such arguments may be made, based on a theory introduced by Le~Cam \cite{lecam}.

We will work again with sequences of probability measures $\PPP = (\PP_n)_{n \in \NN}$ and $\QQQ = (\QQ_n)_{n \in \NN}$, and will denote by $L_n$ the likelihood ratio $d \PP_n/d \QQ_n$. Norms and inner products of functions are those of $L^2(\QQ_n)$.
The following is the crucial definition underlying the arguments to come.

\begin{definition}
    A sequence $\PPP$ of probability measures is \emph{contiguous} to a sequence $\QQQ$, written $\PPP \triangleleft \QQQ$, if whenever $A_n \in \sF_n$ with $\QQ_n(A_n) \to 0$ (as $n \to \infty$), then $\PP_n(A_n) \to 0$ as well.
\end{definition}

\begin{proposition}
    If $\PPP \triangleleft \QQQ$ or $\QQQ \triangleleft \PPP$, then $\QQQ$ and $\PPP$ are statistically indistinguishable (in the sense of Definition~\ref{def:stat-ind}, i.e., no test can have both type I and type II error probabilities tending to 0).
\end{proposition}
\begin{proof}
    We give the proof for the case $\PPP \triangleleft \QQQ$, but the other case may be shown by a symmetric argument.
    For the sake of contradiction, let $(f_n)_{n \in \NN}$ be a sequence of tests distinguishing $\PPP$ and $\QQQ$, and let $A_n = \{\bY: f_n(\bY) = \tp\}$.
    Then, $\PP_n(A_n^c) \to 0$ and $\QQ_n(A_n) \to 0$.
    But, by contiguity, $\QQ_n(A_n) \to 0$ implies $\PP_n(A_n) \to 0$ as well, so $\PP_n(A_n^c) \to 1$, a contradiction.
\end{proof}
\noindent
It therefore suffices to establish contiguity in order to prove negative results about statistical distinguishability.
The following classical second moment method gives a means of establishing contiguity through a computation with the likelihood ratio.
\begin{lemma}[Second Moment Method for Contiguity]
    \label{lem:second-moment-contiguity}
    If $\|L_n\|^2 \colonequals \EE_{\bY \sim \QQ_n}[L_n(\bY)^2]$ remains bounded as $n \to \infty$ (i.e., $\limsup_{n \to \infty} \|L_n\|^2 < \infty$), then $\PPP \triangleleft \QQQ$.
\end{lemma}
\begin{proof}
    Let $A_n \in \sF_n$.
    Then, using the Cauchy--Schwarz inequality,
    \begin{equation*}
        \PP_n(A_n) = \Ex_{\bY \sim \PP_n} [\One_{A_n}(\bY)] = \Ex_{\bY \sim \QQ_n}\left[ L_n(\bY) \One_{A_n}(\bY)\right] \leq \left(\Ex_{\bY \sim \QQ_n}[L_n(\bY)^2]\right)^{1/2} \left(\QQ_n(A_n)\right)^{1/2},
    \end{equation*}
    and so $\QQ_n(A_n) \to 0$ implies $\PP_n(A_n) \to 0$.
\end{proof}

\noindent This second moment method has been used to establish contiguity for various high-dimensional statistical problems (see e.g., \cite{MRZ-spectral,BMVVX-pca,PWBM-pca,PWB-tensor}). Typically the null hypothesis $\QQ_n$ is a ``simpler'' distribution than $\PP_n$ and, as a result, $d\PP_n/d\QQ_n$ is easier to compute than $d\QQ_n/d\PP_n$. In general, and essentially for this reason, establishing $\QQQ \triangleleft \PPP$ is often more difficult than $\PPP \triangleleft \QQQ$, requiring tools such as the \emph{small subgraph conditioning method} (introduced in \cite{subgraph-1,subgraph-2} and used in, e.g., \cite{MNS-rec,BMNN-community}). Fortunately, one-sided contiguity $\PP_n \triangleleft \QQ_n$ is sufficient for our purposes.

Note that $\|L_n\|$, the quantity that controls contiguity per the second moment method, is the same as the optimal value of the $L^2$ optimization problem in Proposition~\ref{prop:lr-optimal-l2}:
\begin{equation*}
    \left\{\begin{array}{rl}
        \text{maximize} & \EE_{\bY \sim \PP_n} [f(\bY)] \\[5pt]
        \text{subject to} & \EE_{\bY \sim \QQ_n} [f(\bY)^2] = 1
    \end{array}\right\} = \|L_n\|.
\end{equation*}
We might then be tempted to conjecture that $\PPP$ and $\QQQ$ are statistically distinguishable \emph{if and only if} $\|L_n\| \to \infty$ as $n \to \infty$. However, this is incorrect: there are cases when $\PPP$ and $\QQQ$ are not distinguishable, yet a rare ``bad'' event under $\PP_n$ causes $\|L_n\|$ to diverge.
To overcome this failure of the ordinary second moment method, some previous works (e.g., \cite{BMNN-community,BMVVX-pca,PWB-tensor,PWBM-pca}) have used \emph{conditional} second moment methods to show indistinguishability, where the second moment method is applied to a modified $\PPP$ that conditions on these bad events not occurring.

\subsection{Basics of the Low-Degree Method}
\label{sec:ldlr-conj}

We now describe the \emph{low-degree} analogues of the notions described in the previous section, which together constitute a method for restricting the classical decision-theoretic second moment analysis to computationally-bounded tests.
The premise of this \emph{low-degree method} is to take low-degree multivariate polynomials in the entries of the observation $\bY$ as a proxy for efficiently-computable functions. The ideas in this section were first developed in a sequence of works in the sum-of-squares optimization literature \cite{pcal,HS-bayesian,sos-hidden,sam-thesis}.

In the computationally-unbounded case, Proposition~\ref{prop:lr-optimal-l2} showed that the likelihood ratio optimally distinguishes $\PPP$ from $\QQQ$ in the $L^2$ sense.
Following the same heuristic, we will now find the low-degree polynomial that best distinguishes $\PPP$ from $\QQQ$ in the $L^2$ sense. In order for polynomials to be defined, we assume here that $\mathcal{S}_n \subseteq \RR^N$ for some $N = N(n)$, i.e., our data (drawn from $\PP_n$ or $\QQ_n$) is a real-valued vector (which may be structured as a matrix, tensor, etc.).

\begin{definition}\label{def:LDLR}
    Let $\sV^{\leq D}_n \subset L^2(\QQ_n)$ denote the linear subspace of polynomials $\sS_n \to \RR$ of degree at most $D$.
    Let $\sP^{\leq D}: L^2(\QQ_n) \to \sV^{\leq D}_n$ denote the orthogonal projection\footnote{To clarify, orthogonal projection is with respect to the inner product induced by $\QQ_n$ (see Definition~\ref{def:hilbert}).} operator to this subspace.
    Finally, define the \emph{$D$-low-degree likelihood ratio ($D$-LDLR)} as $L_n^{\leq D} \colonequals \sP^{\leq D} L_n$.
\end{definition}
\noindent We now have a low-degree analogue of Proposition~\ref{prop:lr-optimal-l2}, which first appeared in \cite{HS-bayesian,sos-hidden}.
\begin{proposition}
    The unique solution $f^*$ of the optimization problem
    \begin{equation}
    \label{eq:l2-opt-low}
    \begin{array}{ll}
        \text{maximize} & \displaystyle\Ex_{\bY \sim \PP_n} [f(\bY)] \\[10pt]
        \text{subject to} & \displaystyle\Ex_{\bY \sim \QQ_n} [f(\bY)^2] = 1, \\[10pt] & f \in \sV_n^{\leq D},
    \end{array}
    \end{equation}
    is the (normalized) $D$-LDLR
    \[
        f^\star = L^{\le D}_n/\|L^{\le D}_n\|,
    \]
    and the value of the optimization problem is $\|L_n^{\leq D}\|$.
\end{proposition}
\begin{proof}
As in the proof of Proposition~\ref{prop:lr-optimal-l2}, we can restate the optimization problem as maximizing $\langle L_n,f \rangle$ subject to $\|f\| = 1$ and $f \in \sV_n^{\le D}$. Since $\sV_n^{\leq D}$ is a linear subspace of $L^2(\QQ_n)$, the result is then simply a restatement of the variational description and uniqueness of the orthogonal projection in $L^2(\QQ_n)$ (i.e., the fact that $L_n^{\le D}$ is the unique closest element of $\sV_n^{\le D}$ to $L_n$).
\end{proof}

The following informal conjecture is at the heart of the low-degree method. It states that a computational analogue of the second moment method for contiguity holds, with $L_n^{\leq D}$ playing the role of the likelihood ratio. Furthermore, it postulates that polynomials of degree roughly $\log(n)$ are a proxy for polynomial-time algorithms. This conjecture is based on \cite{HS-bayesian,sos-hidden,sam-thesis}, particularly Conjecture~2.2.4 of \cite{sam-thesis}.

\begin{conjecture}[Informal]
    \label{conj:low-deg-informal}
    For ``sufficiently nice'' sequences of probability measures $\PPP$ and $\QQQ$, if there exists $\varepsilon > 0$ and $D = D(n) \ge (\log n)^{1+\varepsilon}$ for which $\|L_n^{\leq D}\|$ remains bounded as $n \to \infty$, then there is no polynomial-time algorithm that strongly distinguishes (see Definition~\ref{def:stat-ind}) $\PPP$ and $\QQQ$.
\end{conjecture}
\noindent We will discuss this conjecture in more detail later (see Section~\ref{sec:ldlr-conj-2}), including the informal meaning of ``sufficiently nice'' and a variant of the LDLR based on \emph{coordinate degree} considered by \cite{sos-hidden,sam-thesis} (see Section~\ref{sec:discuss-conj}). A more general form of the low-degree conjecture (Hypothesis~2.1.5 of \cite{sam-thesis}) states that degree-$D$ polynomials are a proxy for time-$n^{\tilde\Theta(D)}$ algorithms, allowing one to probe a wide range of time complexities. We will see that the converse of these low-degree conjectures often holds in practice; i.e., if $\|L_n^{\le D}\| \to \infty$, then there exists a distinguishing algorithm of runtime roughly $n^D$. As a result, the behavior of $\|L_n^{\le D}\|$ precisely captures the (conjectured) power of computationally-bounded testing in many settings.

The remainder of these notes is organized as follows.
In Section~\ref{sec:agn}, we work through the calculations of $L_n$, $L_n^{\leq D}$, and their norms for a general family of additive Gaussian noise models. In Section~\ref{sec:examples}, we apply this analysis to a few specific models of interest: the spiked Wigner matrix and spiked Gaussian tensor models.
In Section~\ref{sec:ldlr-conj-2}, we give some further discussion of Conjecture~\ref{conj:low-deg-informal}, including evidence (both heuristic and formal) in its favor.

\section{The Additive Gaussian Noise Model}
\label{sec:agn}

We will now describe a concrete class of hypothesis testing problems and analyze them using the machinery introduced in the previous section.
The examples we discuss later (spiked Wigner matrix and spiked tensor) will be specific instances of this general class.

\subsection{The Model}

\begin{definition}[Additive Gaussian Noise Model]
\label{def:agn}
Let $N = N(n) \in \NN$ and let $\bX$ (the ``signal'') be drawn from some distribution $\sP_n$ (the ``prior'') over $\RR^N$. Let $\bZ \in \RR^N$ (the ``noise'') have i.i.d.\ entries distributed as $\sN(0,1)$. Then, we define $\PPP$ and $\QQQ$ as follows.
\begin{itemize}
    \item Under $\PP_n$, observe $\bY = \bX + \bZ$.
    \item Under $\QQ_n$, observe $\bY = \bZ$.
\end{itemize}
\end{definition}
\noindent 
One typical situation takes $\bX$ to be a low-rank matrix or tensor. The following is a particularly important and well-studied special case, which we will return to in Section~\ref{sec:spiked-matrix}.

\begin{example}[Wigner Spiked Matrix Model]
\label{ex:wig}
Consider the additive Gaussian noise model with $N = n^2$, $\RR^N$ identified with $n \times n$ matrices with real entries, and $\sP_n$ defined by $\bX = \lambda \bx \bx^{\top} \in \RR^{n \times n}$, where $\lambda = \lambda(n) > 0$ is a signal-to-noise parameter and $\bx$ is drawn from some distribution $\sX_n$ over $\RR^n$. Then, the task of distinguishing $\PP_{n}$ from $\QQ_n$ amounts to distinguishing $\lambda \bx\bx^\top + \bZ$ from $\bZ$ where $\bZ \in \RR^{n \times n}$ has i.i.d.\ entries distributed as $\sN(0, 1)$. (This variant is equivalent to the more standard model in which the noise matrix is symmetric; see Appendix~\ref{app:symm}.)
\end{example}

\noindent
This problem is believed to exhibit stat-comp gaps for some choices of $\sX_n$ but not others; see, e.g., \cite{LKZ-mmse,LKZ-sparse,mi-rank-one,BMVVX-pca,PWBM-pca}.
At a heuristic level, the typical \emph{sparsity} of vectors under $\sX_n$ seems to govern the appearance of a stat-comp gap.

\begin{remark}
In the spiked Wigner problem, as in many others, one natural statistical task besides distinguishing the null and planted models is to non-trivially estimate the vector $\bx$ given $\bY \sim \PP_n$, i.e., to compute an estimate $\hat \bx = \hat \bx (\bY)$ such that $|\langle \hat \bx, \bx \rangle|/(\|\hat \bx\| \cdot \|\bx\|) \ge \varepsilon$ with high probability, for some constant $\varepsilon > 0$. Typically, for natural high-dimensional problems, non-trivial estimation of $\bx$ is statistically or computationally possible precisely when it is statistically or computationally possible (respectively) to strongly distinguish $\PPP$ and $\QQQ$; see Section~\ref{sec:extensions} for further discussion.
\end{remark}

\subsection{Computing the Classical Quantities}

We now show how to compute the likelihood ratio and its $L^2$-norm under the additive Gaussian noise model. (This is a standard calculation; see, e.g., \cite{MRZ-spectral,BMVVX-pca}.)

\begin{proposition}
    \label{prop:agn-L}
    Suppose $\PPP$ and $\QQQ$ are as defined in Definition~\ref{def:agn}, with a sequence of prior distributions $(\sP_n)_{n \in \NN}$.
    Then, the likelihood ratio of $\PP_n$ and $\QQ_n$ is
    \begin{equation*}
        L_n(\bY) = \frac{d\PP_n}{d\QQ_n}(\bY) = \Ex_{\bX \sim \sP_n}\left[ \exp\left(-\frac{1}{2}\|\bX\|^2 + \la \bX, \bY \ra \right)\right].
    \end{equation*}
\end{proposition}
\begin{proof}
    Write $\sL$ for the Lebesgue measure on $\RR^N$.
    Then, expanding the gaussian densities,
    \begin{align}
        \frac{d\QQ_n}{d\sL}(\bY) 
        &= (2\pi)^{-N / 2}\cdot \exp\left(-\frac{1}{2}\|\bY\|^2\right) \label{eq:agn-ln-denom}\\
        \frac{d\PP_n}{d\sL}(\bY) 
        &= (2\pi)^{-N / 2}\cdot \Ex_{\bX \sim \sP_n}\left[\exp\left(-\frac{1}{2}\|\bY - \bX\|^2\right)\right] \nonumber \\
        &= (2\pi)^{-N / 2}\cdot \exp\left(-\frac{1}{2}\|\bY\|^2\right) \cdot \EE_{\bX \sim \sP_n}\left[\exp\left(-\frac{1}{2}\|\bX\|^2 + \la \bX, \bY \ra\right)\right]\label{eq:agn-ln-num},
    \end{align}
    and $L_n$ is given by the quotient of \eqref{eq:agn-ln-num} and \eqref{eq:agn-ln-denom}.
\end{proof}

\begin{proposition}
\label{prop:agn-L-norm}
Suppose $\PPP$ and $\QQQ$ are as defined in Definition~\ref{def:agn}, with a sequence of prior distributions $(\sP_n)_{n \in \NN}$.
Then,
\begin{equation}\label{eq:gaussian-2nd}
\|L_n\|^2 = \Ex_{\bX^1, \bX^2 \sim \sP_n} \exp(\langle \bX^1, \bX^2 \rangle),
\end{equation}
where $\bX^1, \bX^2$ are drawn independently from $\sP_n$.
\end{proposition}
\begin{proof}
We apply the important trick of rewriting a squared expectation as an expectation over the two independent ``replicas'' $\bX^1, \bX^2$ appearing in the result:
\begin{align*}
\|L_n\|^2 &= \Ex_{\bY \sim \QQ_n} \left[\left(\Ex_{\bX \sim \sP_n} \exp\left(\langle \bY,\bX \rangle - \frac{1}{2} \|\bX\|^2\right)\right)^2\right] \\
&= \Ex_{\bY \sim \QQ_n} \Ex_{\bX^1, \bX^2 \sim \sP_n} \exp\left(\langle \bY,\bX^1 + \bX^2\rangle - \frac{1}{2} \|\bX^1\|^2 - \frac{1}{2} \|\bX^2\|^2\right),
\intertext{where $\bX^1$ and $\bX^2$ are drawn independently from $\sP_n$. We now swap the order of the expectations,}
&= \Ex_{\bX^1,\bX^2 \sim \sP_n} \left[\exp\left(-\frac{1}{2} \|\bX^1\|^2 - \frac{1}{2} \|\bX^2\|^2\right)\Ex_{\bY \sim \QQ_n}\exp\left(\langle \bY,\bX^1 + \bX^2\rangle \right)\right],
\intertext{and the inner expectation may be evaluated explicitly using the moment-generating function of a Gaussian distribution (if $y \sim \sN(0,1)$, then for any fixed $t \in \RR$, $\EE[\exp(ty)] = \exp(t^2/2)$),}
&= \Ex_{\bX^1,\bX^2} \exp\left(-\frac{1}{2} \|\bX^1\|^2 - \frac{1}{2} \|\bX^2\|^2 + \frac{1}{2}\|\bX^1 + \bX^2\|^2\right),
\end{align*}
from which the result follows by expanding the term inside the exponential.\footnote{Two techniques from this calculation are elements of the ``replica method'' from statistical physics: (1) writing a power of an expectation as an expectation over independent ``replicas'' and (2) changing the order of expectations and evaluating the moment-generating function. The interested reader may see \cite{MPV-spin-glass} for an early reference, or \cite{MM-IPC,BPW-phys-notes} for two recent presentations.}
\end{proof}

To apply the second moment method for contiguity, it remains to show that~\eqref{eq:gaussian-2nd} is $O(1)$ using problem-specific information about the distribution $\sP_n$. For spiked matrix and tensor models, various general-purpose techniques for doing this are given in \cite{PWBM-pca,PWB-tensor}.

\subsection{Computing the Low-Degree Quantities}
\label{sec:agn-low-deg}

In this section, we will show that the norm of the LDLR (see Section~\ref{sec:ldlr-conj}) takes the following remarkably simple form under the additive Gaussian noise model.
\begin{theorem}
    \label{thm:agn-ldlr-norm}
    Suppose $\PPP$ and $\QQQ$ are as defined in Definition~\ref{def:agn}, with a sequence of prior distributions $(\sP_n)_{n \in \NN}$.
    Let $L_n^{\leq D}$ be as in Definition~\ref{def:LDLR}.
    Then,
    \begin{equation}\label{eq:gaussian-2nd-low}
        \|L_n^{\leq D}\|^2 = \Ex_{\bX^1, \bX^2 \sim \sP_n}\left[ \sum_{d = 0}^D \frac{1}{d!} \la \bX^1, \bX^2 \ra^d \right],
    \end{equation}
    where $\bX^1, \bX^2$ are drawn independently from $\sP_n$.
\end{theorem}

\begin{remark}\label{rem:taylor}
Note that~\eqref{eq:gaussian-2nd-low} can be written as $\EE_{\bX^1,\bX^2} [\exp^{\le D}(\langle \bX^1,\bX^2 \rangle)]$, where $\exp^{\le D}(t)$ denotes the degree-$D$ truncation of the Taylor series of $\exp(t)$. This can be seen as a natural low-degree analogue of the full second moment~\eqref{eq:gaussian-2nd}. However, the low-degree Taylor series truncation in $\exp^{\le D}$ is conceptually distinct from the low-degree projection in $L_n^{\le D}$, because the latter corresponds to truncation in the Hermite orthogonal polynomial basis (see below), while the former corresponds to truncation in the monomial basis.
\end{remark}

\noindent Our proof of Theorem~\ref{thm:agn-ldlr-norm} will follow the strategy of \cite{HS-bayesian,sos-hidden,sam-thesis} of expanding $L_n$ in a basis of orthogonal polynomials with respect to $\QQ_n$, which in this case are the \emph{Hermite polynomials}.

We first give a brief and informal review of the multivariate Hermite polynomials (see Appendix~\ref{app:hermite} or the reference text \cite{Szego-OP} for further information).
The univariate Hermite polynomials\footnote{We will not actually use the definition of the univariate Hermite polynomials (although we will use certain properties that they satisfy as needed), but the definition is included for completeness in Appendix~\ref{app:hermite}.} are a sequence $h_k(x) \in \RR[x]$ for $k \geq 0$, with $\deg h_k = k$.
They may be normalized as $\what{h}_k(x) = h_k(x) / \sqrt{k!}$, and with this normalization satisfy the orthonormality conditions
\begin{equation}
    \label{eq:hermite-orth-1d}
    \Ex_{y \sim \sN(0, 1)}\what{h}_k(y)\what{h}_\ell(y) = \delta_{k\ell}.
\end{equation}
The multivariate Hermite polynomials in $N$ variables are indexed by $\bm\alpha \in \NN^N$, and are merely products of the $h_k$: $H_{\bm\alpha}(\bx) = \prod_{i = 1}^N h_{\alpha_i}(x_i)$.
They also admit a normalized variant $\what{H}_{\bm\alpha}(\bx) = \prod_{i = 1}^N \what{h}_{\alpha_i}(x_i)$, and with this normalization satisfy the orthonormality conditions
\begin{equation*}
    \Ex_{\bY \sim \sN(0, \bm I_N)}\what{H}_{\bm\alpha}(\bY)\what{H}_{\bm\beta}(\bY) = \delta_{\bm\alpha \bm\beta},
\end{equation*}
which may be inferred directly from \eqref{eq:hermite-orth-1d}.

The collection of those $\what{H}_{\bm\alpha}$ for which $|\bm\alpha| \colonequals \sum_{i = 1}^N \alpha_i \leq D$ form an orthonormal basis for $\sV_n^{\leq D}$ (which, recall, is the subspace of polynomials of degree $\le D$).
Thus we may expand
\begin{equation}
    \label{eq:agn-Ln-expansion}
    L_n^{\leq D}(\bY) = \sum_{\substack{\bm\alpha \in \NN^N \\ |\bm\alpha| \leq D}} \la L_n, \what{H}_{\bm\alpha}\ra \what{H}_{\bm\alpha}(\bY) = \sum_{\substack{\bm\alpha \in \NN^N \\ |\bm\alpha| \leq D}} \frac{1}{\prod_{i = 1}^N \alpha_i!} \la L_n, H_{\bm\alpha}\ra H_{\bm\alpha}(\bY),
\end{equation}
and in particular we have
\begin{equation}
    \label{eq:agn-Ln-norm-expansion}
    \|L_n^{\leq D}\|^2 = \sum_{\substack{\bm\alpha \in \NN^N \\ |\bm\alpha| \leq D}} \frac{1}{\prod_{i = 1}^N \alpha_i!} \la L_n, H_{\bm\alpha}\ra^2.
\end{equation}
Our main task is then to compute quantities of the form $\la L_n, H_{\bm\alpha} \ra$. Note that these can be expressed either as $\EE_{\bY \sim \QQ_n}[L_n(\by) H_{\bm\alpha}(\bY)]$ or $\EE_{\bY \sim \PP_n}[H_{\bm\alpha}(\bY)]$.
We will give three techniques for carrying out this calculation, each depending on a different identity satisfied by the Hermite polynomials. Each will give a proof of the following remarkable formula, which shows that the quantities $\la L_n, H_{\bm\alpha} \ra$ are simply the moments of $\sP_n$.

\begin{proposition}
    \label{prop:agn-components}
    For any $\bm\alpha \in \NN^N$,
    \begin{equation*}
        \langle L_n, H_{\bm\alpha} \rangle = \Ex_{\bX \sim \sP_n}\left[ \prod_{i = 1}^N X_i^{\alpha_i} \right].
    \end{equation*}
\end{proposition}

\noindent Before continuing with the various proofs of Proposition~\ref{prop:agn-components}, let us show how to use it to complete the proof of Theorem~\ref{thm:agn-ldlr-norm}.

\begin{proof}[Proof of Theorem~\ref{thm:agn-ldlr-norm}]
    By Proposition~\ref{prop:agn-components} substituted into \eqref{eq:agn-Ln-norm-expansion}, we have
    \begin{align*}
        \|L_n^{\leq D}\|^2
        &= \sum_{\substack{\bm\alpha \in \NN^N \\ |\bm\alpha| \leq D}} \frac{1}{\prod_{i = 1}^N \alpha_i!} \left(\Ex_{\bX \sim \sP_n}\left[\prod_{i = 1}^N X_i^{\alpha_i}\right]\right)^2, \\
        \intertext{and performing the ``replica'' manipulation (from the proof of Proposition~\ref{prop:agn-L-norm}) again, this may be written}
        &= \Ex_{\bX^1, \bX^2 \sim \sP_n}\left[\sum_{\substack{\bm\alpha \in \NN^N \\ |\bm\alpha| \leq D}} \frac{1}{\prod_{i = 1}^N \alpha_i!} \prod_{i = 1}^N (X^1_iX^2_i)^{\alpha_i}\right] \\
        &= \Ex_{\bX^1, \bX^2 \sim \sP_n}\left[\sum_{d = 0}^D \frac{1}{d!}\sum_{\substack{\bm\alpha \in \NN^N \\ |\bm\alpha| = d}} \binom{d}{\alpha_1 \cdots \alpha_N} \prod_{i = 1}^N (X^1_iX^2_i)^{\alpha_i}\right] \\
        &= \Ex_{\bX^1, \bX^2 \sim \sP_n}\left[\sum_{d = 0}^D \frac{1}{d!}\la \bX^1, \bX^2 \ra^d \right],
    \end{align*}
    where the last step uses the multinomial theorem.
\end{proof}

We now proceed to the three proofs of Proposition~\ref{prop:agn-components}. For the sake of brevity, we omit here the (standard) proofs of the three Hermite polynomial identities these proofs are based on, but the interested reader may review those proofs in Appendix~\ref{app:hermite}.

\subsubsection{Proof 1: Hermite Translation Identity}

The first (and perhaps simplest) approach to proving Proposition~\ref{prop:agn-components} uses the following formula for the expectation of a Hermite polynomial evaluated on a Gaussian random variable of non-zero mean.

\begin{proposition}
    \label{prop:hermite-translation}
    For any $k \geq 0$ and $\mu \in \RR$,
    \[ \Ex_{y \sim \sN(\mu, 1)}\left[h_k(y)\right] = \mu^k. \]
\end{proposition}

\begin{proof}[Proof 1 of Proposition~\ref{prop:agn-components}]
    We rewrite $\la L_n, H_{\bm\alpha} \ra$ as an expectation with respect to $\PP_n$:
    \begin{align*}
    \la L_n, H_{\bm\alpha} \ra 
    &= \Ex_{\bY \sim \QQ_n}\left[L_n(\bY)H_{\bm\alpha}(\bY)\right] \\
    &= \Ex_{\bY \sim \PP_n}\left[H_{\bm\alpha}(\bY)\right] \\
    &= \Ex_{\bY \sim \PP_n}\left[\prod_{i = 1}^N h_{\alpha_i}(Y_i) \right]
    \intertext{and recall $\bY = \bX + \bZ$ for $\bX \sim \sP_n$ and $\bZ \sim \sN(\bm 0, \bm I_N)$ under $\PP_n$,}
    &= \Ex_{\bX \sim \sP_n}\left[\Ex_{\bZ \sim \sN(\bm 0, \bm I_N)}\prod_{i = 1}^N h_{\alpha_i}(X_i + Z_i) \right] \\
    &= \Ex_{\bX \sim \sP_n}\left[\prod_{i = 1}^N\Ex_{z \sim \sN(X_i, 1)}h_{\alpha_i}(z)\right] \\
    &= \Ex_{\bX \sim \sP_n}\left[\prod_{i = 1}^N X_i^{\alpha_i}\right],
    \end{align*}
    where we used Proposition~\ref{prop:hermite-translation} in the last step.
\end{proof}

\subsubsection{Proof 2: Gaussian Integration by Parts}

The second approach to proving Proposition~\ref{prop:agn-components} uses the following generalization of a well-known integration by parts formula for Gaussian random variables.
\begin{proposition}
    \label{prop:gaussian-ibp}
    If $f: \RR \to \RR$ is $k$ times continuously differentiable and $f(y)$ and its first $k$ derivatives are bounded by $O(\exp(|y|^\alpha))$ for some $\alpha \in (0, 2)$, then
    \begin{equation*}
        \Ex_{y \sim \sN(0, 1)}\left[h_k(y)f(y)\right] = \Ex_{y \sim \sN(0, 1)}\left[\frac{d^k f}{dy^k}(y)\right].
    \end{equation*}
\end{proposition}
\noindent
(The better-known case is $k = 1$, where one may substitute $h_1(x) = x$.)

\begin{proof}[Proof 2 of Proposition~\ref{prop:agn-components}]
We simplify using Proposition~\ref{prop:gaussian-ibp}:
\begin{align*}
    \la L_n, H_{\bm\alpha} \ra = \Ex_{\bY \sim \QQ_n}\left[L_n(\bY)\prod_{i = 1}^N h_{\alpha_i}(Y_i)\right] = \Ex_{\bY \sim \QQ_n}\left[ \frac{\partial^{|\bm\alpha|}L_n}{\partial Y_1^{\alpha_1} \cdots \partial Y_N^{\alpha_N}}(\bY)\right].
\end{align*}
Differentiating $L_n$ under the expectation, we have
\[ \frac{\partial^{|\bm\alpha|}L}{\partial Y_1^{\alpha_1} \cdots \partial Y_N^{\alpha_N}}(\bY) = \Ex_{\bX \sim \sP_n}\left[\,\prod_{i = 1}^N X_i^{\alpha_i}\, \exp\left(-\frac{1}{2}\|\bX\|^2 + \la \bX, \bY \ra\right)\right]. \]
Taking the expectation over $\bY$, we have $\EE_{\bY \sim \QQ_n} \exp(\la \bX, \bY \ra) = \exp(\frac{1}{2}\|\bX\|^2)$, so the entire second term cancels and the result follows.
\end{proof}

\subsubsection{Proof 3: Hermite Generating Function}

Finally, the third approach to proving Proposition~\ref{prop:agn-components} uses the following generating function for the Hermite polynomials.

\begin{proposition}
    \label{prop:hermite-gf}
    For any $x, y \in \RR$,
    \[ \exp\left(xy - \frac{1}{2}x^2\right) = \sum_{k = 0}^\infty \frac{1}{k!}x^k h_k(y). \]
\end{proposition}

\begin{proof}[Proof 3 of Proposition~\ref{prop:agn-components}]
    We may use Proposition~\ref{prop:hermite-gf} to expand $L_n$ in the Hermite polynomials directly:
    \begin{align*}
        L_n(\bY)
        &= \Ex_{\bX \sim \sP_n}\left[\exp\left(\la \bX, \bY \ra - \frac{1}{2}\|\bX\|^2\right)\right] \\
        &= \Ex_{\bX \sim \sP_n}\left[\,\prod_{i = 1}^N\left(\sum_{k = 0}^\infty\frac{1}{k!}X_i^k h_k(Y_i)\right)\right] \\
        &= \sum_{\bm\alpha \in \NN^N} \frac{1}{\prod_{i = 1}^N \alpha_i!}\, \EE_{\bX \sim \sP_n}\left[\,\prod_{i = 1}^N X_i^{\alpha_i}\right] H_{\bm\alpha}(\bm Y).
    \end{align*}
    Comparing with the expansion \eqref{eq:agn-Ln-expansion} then gives the result.
\end{proof}

Now that we have the simple form~\eqref{eq:gaussian-2nd-low} for the norm of the LDLR, it remains to investigate its convergence or divergence (as $n \to \infty$) using problem-specific statistics of $\bX$. In the next section we give some examples of how to carry out this analysis.

\section{Examples: Spiked Matrix and Tensor Models}
\label{sec:examples}

In this section, we perform the low-degree analysis for a particular important case of the additive Gaussian model: the \emph{order-$p$ spiked Gaussian tensor model}, also referred to as the \emph{tensor PCA (principal component analysis)} problem. This model was introduced by \cite{RM-tensor} and has received much attention recently. The special case $p=2$ of the spiked tensor model is the so-called \emph{spiked Wigner matrix model} which has been widely studied in random matrix theory, statistics, information theory, and statistical physics; see \cite{leo-survey} for a survey.

In concordance with prior work, our low-degree analysis of these models illustrates two representative phenomena: the spiked Wigner matrix model exhibits a sharp computational phase transition, whereas the spiked tensor model (with $p \ge 3$) has a ``soft'' tradeoff between statistical power and runtime which extends through the subexponential-time regime. A low-degree analysis of the spiked tensor model has been carried out previously in \cite{sos-hidden,sam-thesis}; here we give a sharper analysis that more precisely captures the power of subexponential-time algorithms.

In Section~\ref{sec:spiked-tensor}, we carry out our low-degree analysis of the spiked tensor model. In Section~\ref{sec:spiked-matrix}, we devote additional attention to the special case of the spiked Wigner model, giving a refined analysis that captures its sharp phase transition and applies to a variety of distributions of ``spikes.''

\subsection{The Spiked Tensor Model}
\label{sec:spiked-tensor}

We begin by defining the model.

\begin{definition}
An $n$-dimensional order-$p$ \emph{tensor} $\bT \in (\RR^n)^{\otimes p}$ is a multi-dimensional array with $p$ dimensions each of length $n$, with entries denoted $T_{i_1,\ldots,i_p}$ where $i_j \in [n]$. For a vector $\bx \in \RR^n$, the \emph{rank-one tensor} $\bx^{\otimes p} \in (\RR^n)^{\otimes p}$ has entries $(\bx^{\otimes p})_{i_1,\ldots,i_p} = x_{i_1} x_{i_2} \cdots x_{i_p}$.
\end{definition}

\begin{definition}[Spiked Tensor Model]
\label{def:spiked-tensor}
Fix an integer $p \ge 2$. The order-$p$ \emph{spiked tensor model} is the additive Gaussian noise model (Definition~\ref{def:agn}) with $\bX = \lambda \bx^{\otimes p}$, where $\lambda = \lambda(n) > 0$ is a signal-to-noise parameter and $\bx \in \RR^n$ (the ``spike'') is drawn from some probability distribution $\sX_n$ over $\RR^n$ (the ``prior''), normalized so that $\|\bx\|^2 \to n$ in probability as $n \to \infty$. In other words:
\begin{itemize}
    \item Under $\PP_n$, observe $\bY = \lambda \bx^{\otimes p} + \bZ$.
    \item Under $\QQ_n$, observe $\bY = \bZ$.
\end{itemize}
Here, $\bZ$ is a tensor with i.i.d.\ entries distributed as $\sN(0,1)$.\footnote{This model is equivalent to the more standard model in which the noise is symmetric with respect to permutations of the indices; see Appendix~\ref{app:symm}.}
\end{definition}

\noindent Throughout this section we will focus for the sake of simplicity on the Rademacher spike prior, where $\bx$ has i.i.d.\ entries $x_i \sim \Unif(\{\pm 1\})$. We focus on the problem of strongly distinguishing $\PP_n$ and $\QQ_n$ (see Definition~\ref{def:stat-ind}), but, as is typical for high-dimensional problems, the problem of estimating $\bx$ seems to behave in essentially the same way (see Section~\ref{sec:extensions}). 

We first state our results on the behavior of the LDLR for this model.

\begin{theorem}\label{thm:tensor-lowdeg}
Consider the order-$p$ spiked tensor model with $\bx$ drawn from the Rademacher prior, $x_i \sim \Unif(\{\pm 1\})$ i.i.d.\ for $i \in [n]$. Fix sequences $D = D(n)$ and $\lambda = \lambda(n)$. For constants $0 < A_p < B_p$ depending only on $p$, we have the following.\footnote{Concretely, one may take $A_p = \frac{1}{\sqrt{2}} p^{-p/4-1/2}$ and $B_p = \sqrt{2} e^{p/2} p^{-p/4}$.}
\begin{enumerate}
    \item[(i)] If $\lambda \le A_p\, n^{-p/4} D^{(2-p)/4}$ for all sufficiently large $n$, then $\|L_n^{\le D}\| = O(1)$.
    \item[(ii)] If $\lambda \ge B_p\, n^{-p/4} D^{(2-p)/4}$ and $D \le \frac{2}{p}n$ for all sufficiently large $n$, and $D = \omega(1)$, then $\|L_n^{\le D}\| = \omega(1)$.
\end{enumerate}
(Here we are considering the limit $n \to \infty$ with $p$ held fixed, so $O(1)$ and $\omega(1)$ may hide constants depending on $p$.)
\end{theorem}

\noindent Before we prove this, let us interpret its meaning. If we take degree-$D$ polynomials as a proxy for $n^{\tilde\Theta(D)}$-time algorithms (as discussed in Section~\ref{sec:ldlr-conj}), our calculations predict that an $n^{O(D)}$-time algorithm exists when $\lambda \gg n^{-p/4} D^{(2-p)/4}$ but not when $\lambda \ll n^{-p/4} D^{(2-p)/4}$. (Here we ignore log factors, so we use $A \ll B$ to mean $A \le B/\mathrm{polylog}(n)$.) These predictions agree precisely with the previously established statistical-versus-computational tradeoffs in the spiked tensor model! It is known that polynomial-time distinguishing algorithms exist when $\lambda \gg n^{-p/4}$ \cite{RM-tensor,HSS-tensor,tensor-hom,sos-fast}, and sum-of-squares lower bounds suggest that there is no polynomial-time distinguishing algorithm when $\lambda \ll n^{-p/4}$ \cite{HSS-tensor,sos-hidden}. 

Furthermore, one can study the power of subexponential-time algorithms, i.e., algorithms of runtime $n^{n^\delta} = \exp(\tilde{O}(n^{\delta}))$ for a constant $\delta \in (0,1)$. Such algorithms are known to exist when $\lambda \gg n^{-p/4-\delta(p-2)/4}$ \cite{strongly-refuting,BGG,BGL,kikuchi}, matching our prediction.\footnote{Some of these results only apply to minor variants of the spiked tensor problem, but we do not expect this difference to be important.}
These algorithms interpolate smoothly between the polynomial-time algorithm which succeeds when $\lambda \gg n^{-p/4}$, and the exponential-time exhaustive search algorithm which succeeds when $\lambda \gg n^{(1-p)/2}$. (Distinguishing the null and planted distributions is information-theoretically impossible when $\lambda \ll n^{(1-p)/2}$ \cite{RM-tensor,PWB-tensor,tensor-phys,tensor-stat}, so this is indeed the correct terminal value of $\lambda$ for computational questions.) The tradeoff between statistical power and runtime that these algorithms achieve is believed to be optimal, and our results corroborate this claim. Our results are sharper than the previous low-degree analysis for the spiked tensor model \cite{sos-hidden,sam-thesis}, in that we pin down the precise constant $\delta$ in the subexponential runtime. (Similarly precise analyses of the tradeoff between subexponential runtime and statistical power have been obtained for CSP refutation \cite{strongly-refuting} and sparse PCA \cite{subexp-sparse}.)

We now begin the proof of Theorem~\ref{thm:tensor-lowdeg}. Since the spiked tensor model is an instance of the additive Gaussian model, we can apply the formula from Theorem~\ref{thm:agn-ldlr-norm}: letting $\bx^1,\bx^2$ be independent draws from $\sX_n$,
\begin{equation}\label{eq:tensor-L}
\|L_n^{\le D}\|^2 = \Ex_{\bx^1,\bx^2} \exp^{\le D}(\lambda^2 \langle \bx^1,\bx^2 \rangle^p) = \sum_{d=0}^{D} \frac{\lambda^{2d}}{d!} \Ex_{\bx^1,\bx^2} [\langle \bx^1,\bx^2 \rangle^{pd}].
\end{equation}
We will give upper and lower bounds on this quantity in order to prove the two parts of Theorem~\ref{thm:tensor-lowdeg}.

\subsubsection{Proof of Theorem~\ref{thm:tensor-lowdeg}: Upper Bound}

\begin{proof}[Proof of Theorem~\ref{thm:tensor-lowdeg}(i)]

We use the moment bound
\begin{equation}\label{eq:subg-mom}
\Ex_{\bx^1,\bx^2}[|\langle \bx^1,\bx^2 \rangle|^k] \le (2n)^{k/2} k \Gamma(k/2)
\end{equation}
for any integer $k \ge 1$. This follows from  $\langle \bx^1,\bx^2 \rangle$ being a \emph{subgaussian} random variable with variance proxy $n$ (see Appendix~\ref{app:subg} for details on this notion, and see Proposition~\ref{prop:subg-mom} for the bound~\eqref{eq:subg-mom}). Plugging this into~\eqref{eq:tensor-L},
\[ \|L_n^{\le D}\|^2 \le 1 + \sum_{d=1}^D \frac{\lambda^{2d}}{d!} (2n)^{pd/2}pd\,\Gamma(pd/2) =: 1 + \sum_{d=1}^D T_d. \]
Note that $T_1 = O(1)$ provided $\lambda = O(n^{-p/4})$ (which will be implied by~\eqref{eq:tensor-lam} below). Consider the ratio between successive terms:
\[ r_d := \frac{T_{d+1}}{T_d} = \frac{\lambda^2}{d+1} (2n)^{p/2} p \,\frac{\Gamma(p(d+1)/2)}{\Gamma(pd/2)}. \]
Using the bound $\Gamma(x+a)/\Gamma(x) \le (x+a)^a$ for all $a, x > 0$ (see Proposition~\ref{prop:gamma}), we find
\[ r_d \le \frac{\lambda^2}{d+1} (2n)^{p/2} p [p(d+1)/2]^{p/2} \le \lambda^2 p^{p/2+1} n^{p/2} (d+1)^{p/2-1}. \]
Thus if $\lambda$ is small enough, namely if
\begin{equation}\label{eq:tensor-lam}
\lambda \le \frac{1}{\sqrt{2}}\, p^{-p/4-1/2} n^{-p/4} D^{(2-p)/4},
\end{equation}
then $r_d \le 1/2$ for all $1 \le d < D$.
In this case, by comparing with a geometric sum we may bound $\|L_n^{\le D}\|^2 \le 1 + 2 T_1 = O(1)$.
\end{proof}

\subsubsection{Proof of Theorem~\ref{thm:tensor-lowdeg}: Lower Bound}
\label{sec:tensor-lower}

\begin{proof}[Proof of Theorem~\ref{thm:tensor-lowdeg}(ii)]
Note that $\langle \bx^1,\bx^2 \rangle = \sum_{i=1}^n s_i$ where $s_1, \dots, s_n$ are i.i.d.\ Rademacher random variables, so $\EE_{\bx^1,\bx^2}[\langle \bx^1,\bx^2 \rangle^{2k + 1}] = 0$, and
$$\Ex_{\bx^1,\bx^2}[\langle \bx^1,\bx^2 \rangle^{2k}] = \EE\left[\left(\sum_{i=1}^n s_i\right)^{2k}\right] = \sum_{i_1,i_2,\ldots,i_{2k} \in [n]} \EE[s_{i_1} s_{i_2} \cdots s_{i_{2k}}].$$
By counting only the terms $\EE[s_{i_1} s_{i_2} \cdots s_{i_{2k}}]$ in which each $s_i$ appears either 0 or 2 times, we have
\begin{equation}\label{eq:mom-bound}
\Ex_{\bx^1,\bx^2} [\langle \bx^1,\bx^2 \rangle^{2k}] \ge \binom{n}{k}\frac{(2k)!}{2^k}.
\end{equation}
Let $d$ be the largest integer such that $d \le D$ and $pd$ is even. By our assumption $D \le \frac{2}{p}n $, we then have $pd/2 \le n$.
We now bound $\|L_n^{\leq D}\|^2$ by only the degree-$pd$ term of~\eqref{eq:tensor-L}, and using the bounds $\binom{n}{k} \ge (n/k)^k$ (for $1 \le k \le n$) and $(n/e)^n \le n! \le n^n$, we can lower bound that term as follows:
\begin{align*}
\|L_n^{\le D}\|^2 &\ge \frac{\lambda^{2d}}{d!} \Ex_{\bx^1,\bx^2}[\langle \bx^1,\bx^2 \rangle^{pd}] \\
&\ge \frac{\lambda^{2d}}{d!} \binom{n}{pd/2} \frac{(pd)!}{2^{pd/2}} \\
&\ge \frac{\lambda^{2d}}{d^d} \left(\frac{2n}{pd}\right)^{pd/2} \frac{(pd/e)^{pd}}{2^{pd/2}} \\
&= \left(\lambda^2 e^{-p} p^{p/2} n^{p/2} d^{p/2-1} \right)^d.
\end{align*}
Now, if $\lambda$ is large enough, namely if
$$\lambda \ge \sqrt{2} e^{p/2} p^{-p/4} n^{-p/4} D^{(2-p)/4}$$
and $D = \omega(1)$, then $\|L_n^{\le D}\|^2 \ge (2-o(1))^d = \omega(1)$.
\end{proof}

\subsection{The Spiked Wigner Matrix Model: Sharp Thresholds}
\label{sec:spiked-matrix}

We now turn our attention to a more precise understanding of the case $p = 2$ of the spiked tensor model, which is more commonly known as the \emph{spiked Wigner matrix model}.
Our results from the previous section (specialized to $p=2$) suggest that if $\lambda \gg n^{-1/2}$ then there should be a polynomial-time distinguishing algorithm, whereas if $\lambda \ll n^{-1/2}$ then there should not even be a subexponential-time distinguishing algorithm (that is, no algorithm of runtime $\exp(n^{1-\varepsilon})$ for any $\varepsilon > 0$). In this section, we will give a more detailed low-degree analysis that identifies the precise value of $\lambda \sqrt{n}$ at which this change occurs. This type of sharp threshold has been observed in various high-dimensional inference problems; another notable example is the Kesten-Stigum transition for community detection in the stochastic block model \cite{block-model-1,block-model-2,MNS-rec,massoulie,MNS-proof}. It was first demonstrated by \cite{HS-bayesian} that the low-degree method can capture such sharp thresholds.

To begin, we recall the problem setup. Since the interesting regime is $\lambda = \Theta(n^{-1/2})$, we define $\hat\lambda = \lambda \sqrt{2n}$ and take $\hat\lambda$ to be constant (not depending on $n$). With this notation, the spiked Wigner model is as follows:
\begin{itemize}
    \item Under $\PP_n$, observe $\bY = \frac{\hat\lambda}{\sqrt{2n}} \bx\bx^\top + \bZ$ where $\bx \in \RR^n$ is drawn from $\sX_n$.
    \item Under $\QQ_n$, observe $\bY = \bZ$.
\end{itemize}
Here $\bZ$ is an $n \times n$ random matrix with i.i.d.\ entries distributed as $\sN(0,1)$. (This asymmetric noise model is equivalent to the more standard symmetric one; see Appendix~\ref{app:symm}.) We will consider various spike priors $\sX_n$, but require the following normalization.

\begin{assumption}\label{asm:spike-family}
The spike prior $(\sX_n)_{n \in \NN}$ is normalized so that $\bx \sim \sX_n$ satisfies $\|\bx\|^2 \to n$ in probability as $n \to \infty$.
\end{assumption}

\subsubsection{The Canonical Distinguishing Algorithm: PCA}
\label{sec:pca}

There is a simple reference algorithm for testing in the spiked Wigner model, namely \emph{PCA (principal component analysis)}, by which we simply mean thresholding the largest eigenvalue of the (symmetrized) observation matrix.

\begin{definition}
    The \emph{PCA test} for distinguishing $\PPP$ and $\QQQ$ is the following statistical test, computable in polynomial time in $n$.
    Let $\overline{\bY} \colonequals (\bY + \bY^\top) / \sqrt{2n} = \frac{\hat\lambda}{n} \bx\bx^\top + \bW$, where $\bW = (\bZ + \bZ^\top) / \sqrt{2n}$ is a random matrix with the GOE distribution.\footnote{Gaussian Orthogonal Ensemble (GOE): $\bW$ is a symmetric $n \times n$ matrix with entries $W_{ii} \sim \sN(0,2/n)$ and $W_{ij} = W_{ji} \sim \sN(0,1/n)$, independently.}
    Then, let
    \begin{equation*}
        f_{\hat\lambda}^{\mathsf{PCA}}(\bY) \colonequals \left\{\begin{array}{lcl} \tp & : & \lambda_{\max}(\overline{\bY}) > t(\hat\lambda)  \\ \tq & : & \lambda_{\max}(\overline{\bY}) \leq t(\hat\lambda) \end{array}\right\}
    \end{equation*}
    where the threshold is set to  $t(\hat\lambda) \colonequals 2 + (\hat\lambda + \hat\lambda^{-1} - 2) / 2$.
\end{definition}
\noindent
The theoretical underpinning of this test is the following seminal result from random matrix theory, the analogue for Wigner matrices of the celebrated ``BBP transition'' \cite{BBP}.
\begin{theorem}[\cite{FP-bbp,BGN}]
    \label{thm:wig-bbp}
    Let $\hat\lambda$ be constant (not depending on $n$). Let $\overline{\bY} = \frac{\hat\lambda}{n} \bx\bx^\top + \bW$ with $\bW \sim \GOE(n)$ and arbitrary $\bx \in \RR^n$ with $\|\bx\|^2 = n$.
    \begin{itemize}
        \item If $\hat\lambda \leq 1$, then $\lambda_{\max}(\overline{\bY}) \to 2$ as $n \to \infty$ almost surely, and $\la \bv_{\max}(\overline{\bY}), \bx/\sqrt{n} \ra^2 \to 0$ almost surely (where $\lambda_{\max}$ denotes the largest eigenvalue and $\bv_{\max}$ denotes the corresponding unit-norm eigenvector).
        \item If $\hat\lambda > 1$, then $\lambda_{\max}(\overline{\bY}) \to \hat\lambda + \hat\lambda^{-1} > 2$ as $n \to \infty$ almost surely, and $\la \bv_{\max}(\overline{\bY}), \bx/\sqrt{n} \ra^2 \to 1 - \hat\lambda^{-2}$ almost surely.
    \end{itemize}
\end{theorem}

\noindent Thus, the PCA test exhibits a sharp threshold: it succeeds when $\hat\lambda > 1$, and fails when $\hat\lambda \le 1$. (Furthermore, the leading eigenvector achieves non-trivial estimation of the spike $\bx$ when $\hat\lambda > 1$ and fails to do so when $\hat\lambda \le 1$.)

\begin{corollary}
    For any $\hat\lambda > 1$ and any spike prior family $(\sX_n)_{n \in \NN}$ valid per Assumption~\ref{asm:spike-family}, $f_{\hat\lambda}^{\mathsf{PCA}}$ is a polynomial-time statistical test strongly distinguishing $\PPP_{\lambda}$ and $\QQQ$.
\end{corollary}

\noindent For some spike priors $(\sX_n)$, it is known that PCA is statistically optimal, in the sense that distinguishing (or estimating the spike) is information-theoretically impossible when $\hat\lambda < 1$.
These priors include the prior with $\bx$ drawn uniformly from the sphere of radius $\sqrt{n}$ and the priors with $\bx$ having i.i.d.\ $\sN(0, 1)$ or Rademacher (uniformly $\pm 1$) entries \cite{MRZ-spectral,DAM,BMVVX-pca,PWBM-pca}. For these priors, we thus have $\hat\lambda_{\mathsf{stat}} = \hat\lambda_{\mathsf{comp}} = 1$, and there is no stat-comp gap.

A different picture emerges for other spike priors, such as the \emph{sparse Rademacher prior}\footnote{In the sparse Rademacher prior, each entry of $\bx$ is nonzero with probability $\rho$ (independently), and the nonzero entries are drawn uniformly from $\{\pm 1/\sqrt{\rho}\}$.} with constant density $\rho = \Theta(1)$. If $\rho$ is smaller than a particular small constant (roughly $0.09$ \cite{mi-rank-one}), it is known that $\hat\lambda_{\mathsf{stat}} < 1$. More precisely, an exponential-time exhaustive search algorithm succeeds in part of the regime where PCA fails \cite{BMVVX-pca}. For any given $\rho$, the precise threshold $\hat\lambda_{\mathsf{stat}}$ can be computed using the \emph{replica-symmetric formula} from statistical physics (see, e.g., \cite{LKZ-mmse,LKZ-sparse,mi-rank-one,BDMKLZ-spiked,LM-symm,finite-size,replica-short,detection-wig,tensor-stat}, or \cite{leo-survey} for a survey). However, for any constant $\rho$, it is believed that $\hat\lambda_{\mathsf{comp}} = 1$, i.e., that no \emph{polynomial-time} algorithm can ``beat PCA.'' This has been conjectured in the same statistical physics literature based on failure of the \emph{approximate message passing (AMP)} algorithm \cite{LKZ-mmse,LKZ-sparse,mi-rank-one}.

We will next give a low-degree analysis corroborating that conjecture: for a large class of spike priors (including the sparse Rademacher prior with constant $\rho$), our predictions suggest that any distinguishing algorithm requires nearly-exponential time whenever $\hat\lambda < 1$.
One interesting case we will \emph{not} cover is that of the sparse Rademacher prior with $\rho = o(1)$. In such models, there actually \emph{are} subexponential-time algorithms that succeed for some $\hat\lambda < 1$; these are described in \cite{subexp-sparse} along with matching lower bounds using the low-degree method. Furthermore, there are polynomial-time algorithms that can beat the PCA threshold once $\rho \lesssim 1/\sqrt{n}$; this is the much-studied ``sparse PCA'' regime (see, e.g., \cite{JL04,JL09,AW-sparse,BR-sparse,KNV,DM-sparse,subexp-sparse}).

\subsubsection{Low-Degree Analysis: Informally, with the ``Gaussian Heuristic''}
\label{sec:spiked-wigner-gaussian-heuristic}

We first give a heuristic low-degree analysis of the spiked Wigner model which suggests $\hat\lambda_{\mathsf{comp}} = 1$ (matching PCA) for sufficiently ``reasonable'' spike priors $(\sX_n)$. In the next section, we will state and prove a rigorous statement to this effect.

Recall from~\eqref{eq:tensor-L} the expression for the norm of the LDLR:
\begin{equation*}
\|L_n^{\le D}\|^2 = \sum_{d=0}^{D} \frac{\lambda^{2d}}{d!} \Ex_{\bx^1,\bx^2} [\langle \bx^1,\bx^2 \rangle^{2d}] = \sum_{d=0}^{D} \frac{1}{d!}\left(\frac{\hat\lambda^2}{2n}\right)^d \Ex_{\bx^1,\bx^2} [\langle \bx^1,\bx^2 \rangle^{2d}].
\end{equation*}

\noindent To predict $\hat\lambda_{\mathsf{comp}}$, it remains to determine whether $\|L_n^{\le D}\|$ converges or diverges as $n \to \infty$, as a function of $\hat\lambda$.
Recall that, per Assumption~\ref{asm:spike-family}, $\sX_n$ is normalized so that $\|\bx\|^2 \approx n$. Thus when, e.g., $\bx$ has i.i.d.\ entries, we may expect the following informal central limit theorem:
\begin{quote}
    ``When $\bx^1, \bx^2 \sim \sX_n$ independently, $\la \bx^1, \bx^2 \ra$ is distributed approximately as $\sN(0, n)$.''
\end{quote}
Assuming this heuristic applies to the first $2D$ moments of $\la \bx^1, \bx^2 \ra$, and recalling that the Gaussian moments are $\EE_{g \sim \sN(0,1)} [g^{2k}] = (2k-1)!! = \prod_{i = 1}^k (2i - 1)$, we may estimate
\begin{equation*}
    \|L_n^{\leq D}\|^2 \approx \sum_{d=0}^{D} \frac{1}{d!}\left(\frac{\hat\lambda^2}{2n}\right)^d \Ex_{g \sim \sN(0, n)} [g^{2d}] = \sum_{d=0}^{D} \frac{1}{d!}\left(\frac{\hat\lambda^2}{2n}\right)^d n^d(2d - 1)!! =: \sum_{d=0}^D T_d.
\end{equation*}
Imagine $D$ grows slowly with $n$ (e.g., $D \approx \log n$) in order to predict the power of polynomial-time algorithms. The ratio of consecutive terms above is $$\frac{T_{d+1}}{T_d} = \hat\lambda^2\cdot \frac{2d+1}{2(d+1)} \approx \hat\lambda^2,$$
suggesting that $\|L_n^{\le D}\|$ should diverge if $\hat\lambda > 1$ and converge if $\hat\lambda < 1$.

While this style of heuristic analysis is often helpful for guessing the correct threshold, this type of reasoning can break down if $D$ is too large or if $\bx$ is too sparse. In the next section, we therefore give a rigorous analysis of $\|L_n^{\le D}\|$.

\subsubsection{Low-Degree Analysis: Formally, with Concentration Inequalities}
\label{sec:wig-bound-L}

We now give a rigorous proof that $\|L_n^{\leq D}\| = O(1)$ when $\hat\lambda < 1$ (and $\|L_n^{\leq D}\| = \omega(1)$ when $\hat\lambda > 1$), provided the spike prior is ``nice enough.'' Specifically, we require the following condition on the prior.

\begin{definition}\label{def:local-chernoff}
A spike prior $(\sX_n)_{n \in \NN}$ admits a \emph{local Chernoff bound} if for any $\eta > 0$ there exist $\delta > 0$ and $C > 0$ such that for all $n$,
\begin{equation*}
\Pr\left\{|\la \bx^1,\bx^2 \ra| \ge t\right\} \le C \exp\left(-\frac{1}{2n}(1-\eta)t^2\right) \quad \text{for all } t \in [0,\delta n]
\end{equation*}
where $\bx^1,\bx^2$ are drawn independently from $\sX_n$.
\end{definition}
\noindent For instance, any prior with i.i.d.\ \emph{subgaussian} entries admits a local Chernoff bound; see Proposition~\ref{prop:local-chernoff} in Appendix~\ref{app:subg}. This includes, for instance, the sparse Rademacher prior with any constant density $\rho$. The following is the main result of this section, which predicts that for this class of spike priors, any algorithm that beats the PCA threshold requires nearly-exponential time.

\begin{theorem}\label{thm:wig-bound-L}
Suppose $(\sX_n)_{n \in \NN}$ is a spike prior that (i) admits a local Chernoff bound, and (ii) has then $\|\bx\|^2 \le \sqrt{2}n$ almost surely if $\bx \sim \sX_n$. Then, for the spiked Wigner model with $\hat\lambda < 1$ and any $D = D(n) = o(n/\log n)$, we have $\|L_n^{\leq D}\| = O(1)$ as $n \to \infty$.
\end{theorem}

\begin{remark}
The upper bound $\|\bx\|^2 \le \sqrt{2}n$ is without loss of generality (provided $\|\bx\|^2 \to n$ in probability). This is because we can define a modified prior $\tilde \sX_n$ that draws $\bx \sim \sX_n$ and outputs $\bx$ if $\|\bx\|^2 \le \sqrt{2}n$ and $\bm 0$ otherwise. If $(\sX_n)_{n \in \NN}$ admits a local Chernoff bound then so does $(\tilde \sX_n)_{n \in \NN}$. And, if the spiked Wigner model is computationally hard with the prior $(\tilde \sX_n)_{n \in \NN}$, it is also hard with the prior $(\sX_n)_{n \in \NN}$, since the two differ with probability $o(1)$.
\end{remark}

Though we already know that a polynomial-time algorithm (namely PCA) exists when $\hat\lambda > 1$, we can check that indeed $\|L_n^{\le D}\| = \omega(1)$ in this regime. For the sake of simplicity, we restrict this result to the Rademacher prior.
\begin{theorem}\label{thm:wig-above}
Consider the spiked Wigner model with the Rademacher prior: $\bx$ has i.i.d.\ entries $x_i \sim \Unif(\{\pm 1\})$. If $\hat\lambda > 1$, then for any $D = \omega(1)$ we have $\|L_n^{\le D}\| = \omega(1)$.
\end{theorem}
\noindent The proof is a simple modification of the proof of Theorem~\ref{thm:tensor-lowdeg}(ii) in Section~\ref{sec:tensor-lower}; we defer it to Appendix~\ref{app:wig-above}. The remainder of this section is devoted to proving Theorem~\ref{thm:wig-bound-L}.

\begin{proof}[Proof of Theorem~\ref{thm:wig-bound-L}]
Starting from the expression for $\|L_n^{\le D}\|^2$ (see Theorem~\ref{thm:agn-ldlr-norm} and Remark~\ref{rem:taylor}), we split $\|L_n^{\leq D}\|^2$ into two terms, as follows:
$$\|L_n^{\leq D}\|^2 = \Ex_{\bx^1, \bx^2}\left[\exp^{\leq D}\left(\lambda^2\la \bx^1, \bx^2 \ra^2\right)\right]
=: R_1 + R_2,$$
where
\begin{align*}
R_1 &\colonequals \Ex_{\bx^1, \bx^2}\left[\One_{|\langle \bx^1,\bx^2 \rangle| \le \varepsilon n}\,\exp^{\leq D}\left(\lambda^2\la \bx^1, \bx^2 \ra^2\right)\right], \\
R_2 &\colonequals \Ex_{\bx^1, \bx^2}\left[\One_{|\langle \bx^1,\bx^2 \rangle| > \varepsilon n}\,\exp^{\leq D}\left(\lambda^2\la \bx^1, \bx^2 \ra^2\right)\right].
\end{align*}
Here $\varepsilon > 0$ is a small constant to be chosen later. We call $R_1$ the \emph{small deviations} and $R_2$ the \emph{large deviations}, and we will bound these two terms separately.
\paragraph{Bounding the large deviations.}
Using that $\|\bx\|^2 \le \sqrt{2}n$, that $\exp^{\le D}(t)$ is increasing for $t \ge 0$, and the local Chernoff bound (taking $\varepsilon$ to be a sufficiently small constant),
\begin{align*}
R_2 &\le \Pr\left\{|\langle \bx^1,\bx^2 \rangle| > \varepsilon n\right\} \exp^{\leq D}\left(2 \lambda^2 n^2\right)\\
&\le C \exp\left(-\frac{1}{3}\varepsilon^2 n\right) \sum_{d=0}^D \frac{(\hat\lambda^2 n)^d}{d!}
\intertext{and noting that the last term of the sum is the largest since $\hat\lambda^2 n > D$,}
&\le C \exp\left(-\frac{1}{3}\varepsilon^2 n\right) (D+1) \frac{(\hat\lambda^2 n)^D}{D!}\\
&= \exp\left[\log C - \frac{1}{3} \varepsilon^2 n + \log(D+1) + 2D \log \hat\lambda + D \log n - \log(D!)\right]\\
&= o(1)
\end{align*}
provided $D = o(n/\log n)$.

\paragraph{Bounding the small deviations.}

We adapt an argument from \cite{PWB-tensor}. Here we do not need to make use of the truncation to degree $D$ at all, and instead simply use the bound $\exp^{\le D}(t) \le \exp(t)$ for $t \ge 0$.
With this, we bound
\begin{align*}
R_1 &= \Ex_{\bx^1, \bx^2}\left[\One_{|\langle \bx^1,\bx^2 \rangle| \le \varepsilon n}\,\exp^{\leq D}\left(\lambda^2\la \bx^1, \bx^2 \ra^2\right)\right]\\
&\le \Ex_{\bx^1, \bx^2}\left[\One_{|\langle \bx^1,\bx^2 \rangle| \le \varepsilon n}\,\exp\left(\lambda^2\la \bx^1, \bx^2 \ra^2\right)\right]\\
&= \int_0^\infty \Pr\left\{\One_{|\langle \bx^1,\bx^2 \rangle| \le \varepsilon n}\,\exp\left(\lambda^2\la \bx^1, \bx^2 \ra^2\right) \ge u\right\} du\\
&= 1 + \int_1^\infty \Pr\left\{\One_{|\langle \bx^1,\bx^2 \rangle| \le \varepsilon n}\,\exp\left(\lambda^2\la \bx^1, \bx^2 \ra^2\right) \ge u\right\} du\\
&= 1 + \int_0^\infty \Pr\left\{\One_{|\langle \bx^1,\bx^2 \rangle| \le \varepsilon n}\,\langle \bx^1,\bx^2 \rangle^2 \ge t\right\} \lambda^2\exp\left(\lambda^2 t\right) dt \tag{where $\exp\left(\lambda^2 t\right) = u$}\\
&\le 1 + \int_0^\infty C \exp\left(-\frac{1}{2n}(1-\eta)t\right)  \lambda^2\exp\left(\lambda^2 t\right) dt  \tag{using the local Chernoff bound}\\
&\le 1 + \frac{C \hat\lambda^2}{2n} \int_0^\infty \exp\left(-\frac{1}{2n}(1 - \eta - \hat\lambda^2)t\right) dt\\
&= 1 + C \hat\lambda^2 (1 - \eta - \hat\lambda^2)^{-1} \tag{provided $\hat\lambda^2 < 1-\eta$}\\
&= O(1).
\end{align*}
Since $\hat\lambda < 1$, we can choose $\eta > 0$ small enough so that $\hat\lambda^2 < 1 - \eta$, and then choose $\varepsilon$ small enough so that the local Chernoff bound holds. (Here, $\eta$ and $\varepsilon$ depend on $\hat\lambda$ and the spike prior, but not on $n$.)
\end{proof}

\section{More on the Low-Degree Method}
\label{sec:ldlr-conj-2}

In this section, we return to the general considerations introduced in Section~\ref{sec:ldlr-conj} and describe some of the nuances in and evidence for the main conjecture underlying the low-degree method (Conjecture~\ref{conj:low-deg-informal}). Specifically, we investigate the question of what can be concluded (both rigorously and conjecturally) from the behavior of the low-degree likelihood ratio (LDLR) defined in Definition~\ref{def:LDLR}. We present conjectures and formal evidence connecting the LDLR to computational complexity, discussing various caveats and counterexamples along the way. 

In Section~\ref{sec:LDLR-poly}, we explore to what extent the $D$-LDLR controls whether or not degree-$D$ polynomials can distinguish $\PPP$ from $\QQQ$. Then, in Section~\ref{sec:LDLR-alg}, we explore to what extent the LDLR controls whether or not \emph{any} efficient algorithm can distinguish $\PPP$ and $\QQQ$.

\subsection{The LDLR and Thresholding Polynomials}
\label{sec:LDLR-poly}

Heuristically, since $\|L_n^{\le D}\|$ is the value of the $L^2$ optimization problem \eqref{eq:l2-opt-low}, we might expect the behavior of $\|L_n^{\le D}\|$ as $n \to \infty$ to dictate whether or not degree-$D$ polynomials can distinguish $\PPP$ from $\QQQ$: it should be possible to strongly distinguish (in the sense of Definition~\ref{def:stat-ind}, i.e., with error probabilities tending to $0$) $\PPP$ from $\QQQ$ by \emph{thresholding} a degree-$D$ polynomial (namely $L_n^{\le D}$) if and only if $\|L_n^{\le D}\| = \omega(1)$. We now discuss to what extent this heuristic is correct.

\begin{question}\label{q:thresh-poly-yes}
If $\|L_n^{\le D}\| = \omega(1)$, does this imply that it is possible to strongly distinguish $\PPP$ and $\QQQ$ by thresholding a degree-$D$ polynomial?
\end{question}

\noindent We have already mentioned (see Section~\ref{sec:contig}) a counterexample when $D = \infty$: there are cases where $\PPP$ and $\QQQ$ are not statistically distinguishable, yet $\|L_n\| \to \infty$ due to a rare ``bad'' event under $\PP_n$. Examples of this phenomenon are fairly common (e.g., \cite{BMNN-community,BMVVX-pca,PWBM-pca,PWB-tensor}). However, after truncation to only low-degree components, this issue seems to disappear. For instance, in sparse PCA, $\|L_n^{\le D}\| \to \infty$ only occurs when either (i) there actually is an $n^{\tilde{O}(D)}$-time distinguishing algorithm, or (ii) $D$ is ``unreasonably large,'' in the sense that there is a trivial $n^{t(n)}$-time exhaustive search algorithm and $D \gg t(n)$ \cite{subexp-sparse}. Indeed, we do not know any example of a natural problem where $\|L_n^{\le D}\|$ diverges spuriously for a ``reasonable'' value of $D$ (in the above sense), although one can construct unnatural examples by introducing a rare ``bad'' event in $\PP_n$. Thus, it seems that for natural problems and reasonable growth of $D$, the smoothness of low-degree polynomials regularizes $L_n^{\le D}$ in such a way that the answer to Question~\ref{q:thresh-poly-yes} is typically ``yes.'' This convenient feature is perhaps related to the probabilistic phenomenon of \emph{hypercontractivity}; see Appendix~\ref{app:hyp} and especially Remark~\ref{rem:low-dom}.

Another counterexample to Question~\ref{q:thresh-poly-yes} is the following. Take $\PPP$ and $\QQQ$ that are ``easy'' for degree-$D$ polynomials to distinguish, i.e., $\|L_n^{\le D}\| \to \infty$ and $\PPP, \QQQ$ can be strongly distinguished by thresholding a degree-$D$ polynomial. Define a new sequence of ``diluted'' planted measures $\PPP^\prime$ where $\PP^\prime_n$ samples from $\PP_n$ with probability $1/2$, and otherwise samples from $\QQ_n$. Letting $L'_n = d\PP'_n/d\QQ_n$, we have $\|(L'_n)^{\le D}\| \to \infty$, yet $\PPP'$ and $\QQQ$ cannot be strongly distinguished (even statistically). While this example is perhaps somewhat unnatural, it illustrates that a rigorous positive answer to Question~\ref{q:thresh-poly-yes} would need to restrict to $\PPP$ that are ``homogeneous'' in some sense.

Thus, while we have seen some artificial counterexamples, the answer to Question~\ref{q:thresh-poly-yes} seems to typically be ``yes'' for natural high-dimensional problems, so long as $D$ is not unreasonably large. We now turn to the converse question.

\begin{question}\label{q:thresh-poly-no}
If $\|L_n^{\le D}\| = O(1)$, does this imply that it is impossible to strongly distinguish $\PPP$ and $\QQQ$ by thresholding a degree-$D$ polynomial?
\end{question}

\noindent Here, we are able to give a positive answer in a particular formal sense.
The following result addresses the contrapositive of Question~\ref{q:thresh-poly-no}: it shows that distinguishability by thresholding low-degree polynomials implies exponential growth of the norm of the LDLR.

\begin{theorem}\label{thm:thresh-poly-hard}
Suppose $\QQ$ draws $\bY \in \RR^N$ with entries either i.i.d.\ $\sN(0,1)$ or i.i.d.\ $\Unif(\{\pm 1\})$. Let $\PP$ be any measure on $\RR^N$ that is absolutely continuous with respect to $\QQ$. Let $f: \RR^N \to \RR$ be a polynomial of degree $\le d$ satisfying
\begin{equation}\label{eq:strong-thresh}
\Ex_{\bY \sim \PP}[f(\bY)] \ge A \quad \text{and} \quad \QQ(|f(\bY)| \ge B) \le \delta
\end{equation}
for some $A > B > 0$ and some $\delta \le \frac{1}{2} \cdot 3^{-4kd}$. Then for any $k \in \NN$,
$$\|L^{\le 2kd}\| \ge \frac{1}{2}\left(\frac{A}{B}\right)^{2k}.$$
\end{theorem}

\noindent We defer the proof to Appendix~\ref{app:thresh-poly-hard}. To understand what the result shows, imagine, for example, that $A > B$ are both constants, and $k$ grows slowly with $n$ (e.g., $k \approx \log n$). Then, if a degree-$d$ polynomial (where $d$ may depend on $n$) can distinguish $\PPP$ from $\QQQ$ in the sense of \eqref{eq:strong-thresh} with $\delta = \frac{1}{2} \cdot 3^{-4kd}$, then $\|L_n^{\le 2kd}\| \to \infty$ as $n \to \infty$. Note though, that one weakness of this result is that we require the explicit quantitative bound $\delta \leq \frac{1}{2} \cdot 3^{-4kd}$, rather than merely $\delta = o(1)$.

The proof (see Appendix~\ref{app:thresh-poly-hard}) is a straightforward application of \emph{hypercontractivity} (see e.g., \cite{boolean-book}), a type of result stating that random variables obtained by evaluating low-degree polynomials on weakly-dependent distributions (such as i.i.d.\ ones) are well-concentrated and otherwise ``reasonable.''
We have restricted to the case where $\QQ$ is i.i.d.\ Gaussian or Rademacher because hypercontractivity results are most readily available in these cases, but we expect similar results to hold more generally.

\subsection{Algorithmic Implications of the LDLR}
\label{sec:LDLR-alg}

Having discussed the relationship between the LDLR and low-degree polynomials, we now discuss the relationship between low-degree polynomials and the power of \emph{any} computationally-bounded algorithm.

Any degree-$D$ polynomial has at most $n^D$ monomial terms and so can be evaluated in time $n^{O(D)}$ (assuming that the individual coefficients are easy to compute). However, certain degree-$D$ polynomials can of course be computed faster, e.g., if the polynomial has few nonzero monomials or has special structure allowing it to be computed via a spectral method (as in the \emph{color coding} trick \cite{color-coding} used by~\cite{HS-bayesian}).
Despite such special cases, it appears that for average-case high-dimensional hypothesis testing problems, degree-$D$ polynomials are typically as powerful as general $n^{\tilde{\Theta}(D)}$-time algorithms; this informal conjecture appears as Hypothesis~2.1.5 in \cite{sam-thesis}, building on the work of \cite{pcal,HS-bayesian,sos-hidden} (see also our previous discussion in Section~\ref{sec:ldlr-conj}). We will now explain the nuances and caveats of this conjecture, and give evidence (both formal and heuristic) in its favor.

\subsubsection{Robustness}
\label{sec:robust}

An important counterexample that we must be careful about is XOR-SAT. In the random 3-XOR-SAT problem, there are $n$ $\{\pm 1 \}$-valued variables $x_1,\ldots,x_n$ and we are given a formula consisting of $m$ random constraints of the form $x_{i_\ell} x_{j_\ell} x_{k_\ell} = b_\ell$ for $\ell \in [m]$, with $b_\ell \in \{\pm 1\}$. The goal is to determine whether there is an assignment $x \in \{\pm 1\}^n$ that satisfies all the constraints. Regardless of $m$, this problem can be solved in polynomial time using Gaussian elimination over the finite field $\mathbb{F}_2$. However, when $n \ll m \ll n^{3/2}$, the low-degree method nevertheless predicts that the problem should be computationally hard, i.e., it is hard to distinguish between a random formula (which is unsatisfiable with high probability) and a formula with a planted assignment. This pitfall is not specific to the low-degree method: sum-of-squares lower bounds, statistical query lower bounds, and the cavity method from statistical physics also incorrectly suggest the same (this is discussed in \cite{sq-parity,stat-phys-survey,sos-notes}).

The above discrepancy can be addressed (see, e.g., Lecture~3.2 of \cite{sos-notes}) by noting that Gaussian elimination is very brittle, in the sense that it no longer works to search for an assignment satisfying only a $1-\delta$ fraction of the constraints (as in this case it does not seem possible to leverage the problem's algebraic structure over $\mathbb{F}_2$). Another example of a brittle algorithm is the algorithm of \cite{reg-LLL} for linear regression, which uses Lenstra-Lenstra-Lov\'asz lattice basis reduction \cite{LLL} and only tolerates an exponentially-small level of noise. Thus, while there sometimes exist efficient algorithms that are ``high-degree'', these tend not to be robust to even a tiny amount of noise. As with SoS lower bounds, we expect that the low-degree method correctly captures the limits of \emph{robust} hypothesis testing \cite{sos-hidden} for high-dimensional problems. (Here, ``robust'' refers to the ability to handle a small amount of noise, and should not be confused with the specific notion of \emph{robust inference} \cite{sos-hidden} or with other notions of robustness that allow adversarial corruptions \cite{FK-semirandom,robust-est}.)

\subsubsection{Connection to Sum-of-Squares}
\label{sec:sos}

The sum-of-squares (SoS) hierarchy \cite{parrilo-thesis,lasserre} is a hierarchy of increasingly powerful semidefinite programming (SDP) relaxations for general polynomial optimization problems. Higher levels of the hierarchy produce larger SDPs and thus require more time to solve: level $d$ typically requires time $n^{O(d)}$. SoS lower bounds show that certain levels of the hierarchy fail to solve a given problem. As SoS seems to be at least as powerful as all known algorithms for many problems, SoS lower bounds are often thought of as the ``gold standard'' of formal evidence for computational hardness of average-case problems. For instance, if any constant level $d$ of SoS fails to solve a problem, this is strong evidence that no polynomial-time algorithm exists to solve the same problem (modulo the robustness issue discussed above). 

In order to prove SoS lower bounds, one needs to construct a valid primal certificate (also called a \emph{pseudo-expectation}) for the SoS SDP. The \emph{pseudo-calibration} approach \cite{pcal} provides a strategy for systematically constructing a pseudo-expectation; however, showing that the resulting object is valid (in particular, showing that a certain associated matrix is positive semidefinite) often requires substantial work. As a result, proving lower bounds against constant-level SoS programs is often very technically challenging (as in~\cite{pcal,sos-hidden}). We refer the reader to~\cite{sos-survey} for a survey of SoS and pseudo-calibration in the context of high-dimensional inference.

On the other hand, it was observed by the authors of \cite{pcal,sos-hidden} that the bottleneck for the success of the pseudo-calibration approach seems to typically be a simple condition, none other than the boundedness of the norm of the LDLR (see Conjecture~3.5 of \cite{sos-survey} or Section~4.3 of \cite{sam-thesis}).\footnote{More specifically, $(\|L_n^{\le D}\|^2 - 1)$ is the variance of a certain pseudo-expectation value generated by pseudo-calibration, whose actual value in a valid pseudo-expectation must be exactly 1. It appears to be impossible to ``correct'' this part of the pseudo-expectation if the variance is diverging with $n$.} Through a series of works \cite{pcal,HS-bayesian,sos-hidden,sam-thesis}, the low-degree method emerged from investigating this simpler condition in its own right. It was shown in \cite{HS-bayesian} that SoS can be used to achieve sharp computational thresholds (such as the Kesten--Stigum threshold for community detection), and that the success of the associated method also hinges on the boundedness of the norm of the LDLR. So, historically speaking, the low-degree method can be thought of as a ``lite'' version of SoS lower bounds that is believed to capture the essence of what makes SoS succeed or fail (see~\cite{HS-bayesian,sos-hidden,sos-survey,sam-thesis}).

A key advantage of the low-degree method over traditional SoS lower bounds is that it greatly simplifies the technical work required, allowing sharper results to be proved with greater ease. Moreover, the low-degree method is arguably more natural in the sense that it is not specific to any particular SDP formulation and instead seems to capture the essence of what makes problems computationally easy or hard. On the other hand, some would perhaps argue that SoS lower bounds constitute stronger evidence for hardness than low-degree lower bounds (although we do not know any average-case problems for which they give different predictions).

We refer the reader to~\cite{sam-thesis} for more on the relation between SoS and the low-degree method, including evidence for why the two methods are believed to predict the same results.

\subsubsection{Connection to Spectral Methods}
\label{sec:spectral}

For high-dimensional hypothesis testing problems, a popular class of algorithms are the \emph{spectral methods}, algorithms that build a matrix $\bM$ using the data and then threshold its largest eigenvalue. (There are also spectral methods for estimation problems, usually extracting an estimate of the signal from the leading eigenvector of $\bM$.) Often, spectral methods match the best\footnote{Here, ``best'' is in the sense of strongly distinguishing $\PP_n$ and $\QQ_n$ throughout the largest possible regime of model parameters.} performance among all known polynomial-time algorithms. Some examples include the non-backtracking and Bethe Hessian spectral methods for the stochastic block model \cite{spectral-redemption,massoulie,nb-spectrum,bethe-hessian}, the covariance thresholding method for sparse PCA \cite{DM-sparse}, and the tensor unfolding method for tensor PCA \cite{RM-tensor,HSS-tensor}. As demonstrated in \cite{HSS-tensor,sos-fast}, it is often possible to design spectral methods that achieve the same performance as SoS; in fact, some formal evidence indicates that low-degree spectral methods (where each matrix entry is a constant-degree polynomial of the data) are as powerful as any constant-degree SoS relaxation \cite{sos-hidden}\footnote{In \cite{sos-hidden}, it is shown that for a fairly general class of average-case hypothesis testing problems, if SoS succeeds in some range of parameters then there is a low-degree spectral method whose maximum \emph{positive} eigenvalue succeeds (in a somewhat weaker range of parameters). However, the resulting matrix could \emph{a priori} have an arbitrarily large (in magnitude) negative eigenvalue, which would prevent the spectral method from running in polynomial time. For this same reason, it seems difficult to establish a formal connection between SoS and the LDLR via spectral methods.}.

As a result, it is interesting to try to prove lower bounds against the class of spectral methods. Roughly speaking, the largest eigenvalue in absolute value of a polynomial-size matrix $\bM$ can be computed using $O(\log n)$ rounds of power iteration, and thus can be thought of as an $O(\log n)$-degree polynomial; more specifically, the associated polynomial is $\Tr(\bM^{2k})$ where $k \sim \log(n)$. The following result makes this precise, giving a formal connection between the low-degree method and the power of spectral methods. The proof is given in Appendix~\ref{app:spectral-hard} (and is similar to that of Theorem~\ref{thm:thresh-poly-hard}).

\begin{theorem}\label{thm:spectral-hard}
Suppose $\QQ$ draws $\bY \in \RR^N$ with entries either i.i.d.\ $\sN(0,1)$ or i.i.d.\ $\Unif(\{\pm 1\})$. Let $\PP$ be any measure on $\RR^N$ that is absolutely continuous with respect to $\QQ$. Let $\bM = \bM(\bY)$ be a real symmetric $L \times L$ matrix, each of whose entries is a polynomial in $\bY$ of degree $\le d$. Suppose
\begin{equation}\label{eq:strong-thresh-spec}
\Ex_{\bY \sim \PP}\|\bM\| \ge A \quad \text{and} \quad \QQ(\|\bM\| \ge B) \le \delta
\end{equation}
(where $\|\cdot\|$ denotes matrix operator norm) for some $A > B > 0$ and some $\delta \le \frac{1}{2} \cdot 3^{-4kd}$. Then, for any $k \in \NN$,
$$\|L^{\le 2kd}\| \ge \frac{1}{2L}\left(\frac{A}{B}\right)^{2k}.$$
\end{theorem}
\noindent For example, suppose we are interested in polynomial-time spectral methods, in which case we should consider $L = \mathrm{poly}(n)$ and $d = O(1)$. If there exists a spectral method with these parameters that distinguishes $\PPP$ from $\QQQ$ in the sense of \eqref{eq:strong-thresh-spec} for some constants $A > B$, and with $\delta \to 0$ faster than any inverse polynomial (in $n$), then there exists a choice of $k = O(\log n)$ such that $\|L^{\le O(\log n)}\| = \omega(1)$. And, by contrapositive, if we could show that $\|L^{\le D}\| = O(1)$ for some $D = \omega(\log n)$, that would imply that there is no spectral method with the above properties. This justifies the choice of logarithmic degree in Conjecture~\ref{conj:low-deg-informal}. Similarly to Theorem~\ref{thm:thresh-poly-hard}, one weakness of Theorem~\ref{thm:spectral-hard} is that we can only rule out spectral methods whose failure probability is smaller than any inverse polynomial, instead of merely $o(1)$.

\begin{remark}
Above, we have argued that polynomial-time spectral methods correspond to polynomials of degree roughly $\log(n)$. What if we are instead interested in subexponential runtime $\exp(n^\delta)$ for some constant $\delta \in (0,1)$? One class of spectral method computable with this runtime is that where the dimension is $L \approx \exp(n^\delta)$ and the degree of each entry is $d \approx n^\delta$ (such spectral methods often arise based on SoS \cite{strongly-refuting,BGG,BGL}). To rule out such a spectral method using Theorem~\ref{thm:spectral-hard}, we would need to take $k \approx \log(L) \approx n^\delta$ and would need to show $\|L^{\le D}\| = O(1)$ for $D \approx n^{2\delta}$. However, Conjecture~\ref{conj:low-deg-informal} postulates that time-$\exp(n^\delta)$ algorithms should instead correspond to degree-$n^\delta$ polynomials, and this correspondence indeed appears to be the correct one based on the examples of tensor PCA (see Section~\ref{sec:spiked-tensor}) and sparse PCA (see \cite{subexp-sparse}).

Although this seems at first to be a substantial discrepancy, there is evidence that there are actually spectral methods of dimension $L \approx \exp(n^\delta)$ and \emph{constant} degree $d = O(1)$ that achieve optimal performance among $\exp(n^\delta)$-time algorithms. Such a spectral method corresponds to a degree-$n^\delta$ polynomial, as expected. These types of spectral methods have been shown to exist for tensor PCA \cite{kikuchi}.
\end{remark}

\subsubsection{Formal Conjecture}
\label{sec:discuss-conj}

We next discuss the precise conjecture that Hopkins \cite{sam-thesis} offers on the algorithmic implications of the low-degree method. Informally, the conjecture is that for ``sufficiently nice'' $\PPP$ and $\QQQ$, if $\|L_n^{\le D}\| = O(1)$ for some $D \ge (\log n)^{1+\varepsilon}$, then there is no polynomial-time algorithm that strongly distinguishes $\PPP$ and $\QQQ$. We will not state the full conjecture here (see Conjecture~2.2.4 in \cite{sam-thesis}) but we will briefly discuss some of the details that we have not mentioned yet.

Let us first comment on the meaning of ``sufficiently nice'' distributions. Roughly speaking, this means that: 
\begin{enumerate}
    \item $\QQ_n$ is a product distribution,
    \item $\PP_n$ is sufficiently symmetric with respect to permutations of its coordinates, and
    \item $\PP_n$ is then perturbed by a small amount of additional noise.
\end{enumerate}
Conditions (1) and (2) or minor variants thereof are fairly standard in high-dimensional inference problems. The reason for including (3) is to rule out non-robust algorithms such as Gaussian elimination (see Section~\ref{sec:robust}).

One difference between the conjecture of \cite{sam-thesis} and the conjecture discussed in these notes is that \cite{sam-thesis} considers the notion of \emph{coordinate degree} rather than \emph{polynomial degree}. A polynomial has coordinate degree $\le D$ if no monomial involves more than $D$ variables; however, each individual variable can appear with arbitrarily-high degree in a monomial.\footnote{Indeed, coordinate degree need not be phrased in terms of polynomials, and one may equivalently consider the linear subspace of $L^2(\QQ_n)$ of functions that is spanned by functions of at most $D$ variables at a time.} In \cite{sam-thesis}, the low-degree likelihood ratio is defined as the projection of $L_n$ onto the space of polynomials of coordinate degree~$\le D$. The reason for this is to capture, e.g., algorithms that preprocess the data by applying a complicated high-degree function entrywise. However, we are not aware of any natural problem in which it is important to work with coordinate degree instead of polynomial degree. While working with coordinate degree gives lower bounds that are formally stronger, we work with polynomial degree throughout these notes because it simplifies many of the computations.

\subsubsection{Empirical Evidence and Refined Conjecture}

Perhaps the strongest form of evidence that we have in favor of the low-degree method is simply that it has been carried out on many high-dimensional inference problems and seems to always give the correct predictions, coinciding with widely-believed conjectures. These problems include planted clique \cite{sam-thesis} (implicit in \cite{pcal}), community detection in the stochastic block model \cite{HS-bayesian,sam-thesis}, the spiked tensor model \cite{sos-hidden,sam-thesis}, the spiked Wishart model \cite{BKW-sk}, and sparse PCA \cite{subexp-sparse}. In these notes we have also carried out low-degree calculations for the spiked Wigner model and spiked tensor model (see Section~\ref{sec:examples}). Some of the early results \cite{HS-bayesian,sos-hidden} showed only $\|L_n^{\le D}\| = n^{o(1)}$ as evidence for hardness, which was later improved to $O(1)$ \cite{sam-thesis}. Some of the above results \cite{sos-hidden,sam-thesis} use coordinate degree instead of degree (as we discussed in Section~\ref{sec:discuss-conj}). Throughout the above examples, the low-degree method has proven to be versatile in that it can predict both sharp threshold behavior as well as precise smooth tradeoffs between subexponential runtime and statistical power (as illustrated in the two parts of Section~\ref{sec:examples}).

As discussed earlier, there are various reasons to believe that if $\|L_n^{\le D}\| = O(1)$ for some $D = \omega(\log n)$ then there is no polynomial-time distinguishing algorithm; for instance, this allows us to rule out a general class of spectral methods (see Theorem~\ref{thm:spectral-hard}). However, we have observed that in numerous examples, the LDLR actually has the following more precise behavior that does not involve the extra factor of $\log(n)$.
\begin{conjecture}[Informal]
Let $\PPP$ and $\QQQ$ be ``sufficiently nice.'' If there exists a polynomial-time algorithm to strongly distinguish $\PPP$ and $\QQQ$ then $\|L_n^{\le D}\| = \omega(1)$ for any $D = \omega(1)$.
\end{conjecture}

\noindent In other words, if $\|L_n^{\le D}\| = O(1)$ for some $D = \omega(1)$, this already constitutes evidence that there is no polynomial-time algorithm.
The above seems to be a cleaner version of the main low-degree conjecture that remains correct for many problems of practical interest.

\subsubsection{Extensions}
\label{sec:extensions}

While we have focused on the setting of hypothesis testing throughout these notes, we remark that low-degree arguments have also shed light on other types of problems such as \emph{estimation} (or \emph{recovery}) and \emph{certification}.

First, as we have mentioned before, non-trivial estimation\footnote{Non-trivial estimation of a signal $\bx \in \RR^n$ means having an estimator $\hat{\bm x}$ achieving $|\langle \hat \bx, \bx \rangle|/(\|\hat \bx\| \cdot \|\bx\|) \ge \varepsilon$ with high probability, for some constant $\varepsilon > 0$.} typically seems to be precisely as hard as strong distinguishing (see Definition~\ref{def:stat-ind}), in the sense that the two problems share the same $\lambda_{\mathsf{stat}}$ and $\lambda_{\mathsf{comp}}$. For example, the statistical thresholds for testing and recovery are known to coincide for problems such as the two-groups stochastic block model \cite{MNS-rec,massoulie,MNS-proof} and the spiked Wigner matrix model (for a large class of spike priors) \cite{finite-size,detection-wig}. Also, for any additive Gaussian noise model, any lower bound against hypothesis testing using the second moment method (Lemma~\ref{lem:second-moment-contiguity}) or a conditional second moment method also implies a lower bound against recovery \cite{BMVVX-pca}. More broadly, we have discussed (see Section~\ref{sec:spectral}) how suitable spectral methods typically give optimal algorithms for high-dimensional problems; such methods typically succeed at testing and recovery in the same regime of parameters, because whenever the leading eigenvalue undergoes a phase transition, the leading eigenvector will usually do so as well (see Theorem~\ref{thm:wig-bbp} for a simple example). Thus, low-degree evidence that hypothesis testing is hard also constitutes evidence that non-trivial recovery is hard, at least heuristically. Note, however, that there is no formal connection (in either direction) between testing and recovery (see \cite{BMVVX-pca}), and there are some situations in which the testing and recovery thresholds differ (e.g., \cite{planting-trees}).

In a different approach, Hopkins and Steurer \cite{HS-bayesian} use a low-degree argument to study the recovery problem more directly. In the setting of community detection in the stochastic block model, they examine whether there is a low-degree polynomial that can non-trivially estimate whether two given network nodes are in the same community. They show that such a polynomial exists only when the parameters of the model lie above the problem's widely-conjectured computational threshold, the Kesten--Stigum threshold. This constitutes direct low-degree evidence that recovery is computationally hard below the Kesten--Stigum threshold.

A related (and more refined) question is that of determining the optimal estimation error (i.e., the best possible correlation between the estimator and the truth) for any given signal-to-noise parameter $\lambda$. Methods such as \emph{approximate message passing} can often answer this question very precisely, both statistically and computationally (see, e.g., \cite{AMP,LKZ-mmse,DAM,BDMKLZ-spiked}, or \cite{stat-phys-survey,BPW-phys-notes,leo-survey} for a survey). One interesting question is whether one can recover these results using a variant of the low-degree method.

Another type of statistical task is \emph{certification}. Suppose that $\bY \sim \QQ_n$ has some property $\sP$ with high probability. We say an algorithm \emph{certifies} the property $\sP$ if (i) the algorithm outputs ``yes'' with high probability on $\bY \sim \QQ_n$, and (ii) if $\bY$ does not have property $\sP$ then the algorithm \emph{always} outputs ``no.'' In other words, when the algorithm outputs ``yes'' (which is usually the case), this constitutes a proof that $\bY$ indeed has property $\sP$. Convex relaxations (including SoS) are a common technique for certification. In \cite{BKW-sk}, the low-degree method is used to argue that certification is computationally hard for certain structured PCA problems. The idea is to construct a \emph{quiet planting} $\PP_n$, which is a distribution for which (i) $\bY \sim \PP_n$ never has property $\sP$, and (ii) the low-degree method indicates that it is computationally hard to strongly distinguish $\PP_n$ and $\QQ_n$. In other words, this gives a reduction from a hypothesis testing problem to a certification problem, since any certification algorithm can be used to distinguish $\PP_n$ and $\QQ_n$. (Another example of this type of reduction, albeit relying on a different heuristic for computational hardness, is \cite{hard-rip}.)

\section*{Acknowledgments}
\addcontentsline{toc}{section}{Acknowledgments}
We thank the participants of a working group on the subject of these notes, organized by the authors at the Courant Institute of Mathematical Sciences during the spring of 2019.
We also thank Samuel B.\ Hopkins, Philippe Rigollet, and David Steurer for helpful discussions.

\addcontentsline{toc}{section}{References}
\bibliographystyle{alpha}
\bibliography{main}

\appendix

\section{Omitted Proofs}

\subsection{Neyman-Pearson Lemma}
\label{app:neyman-pearson}

We include here, for completeness, a proof of the classical Neyman--Pearson lemma \cite{N-P}.

\begin{proof}[Proof of Lemma~\ref{lem:neyman-pearson}]
    Note first that a test $f$ is completely determined by its \emph{rejection region}, $R_f = \{\bY: f(\bY) = \PP\}$.
    We may rewrite the power of $f$ as
    \begin{equation*}
        1 - \beta(f) = \PP[f(\bY) = \PP] = \int_{R_f}d\PP(\bY) = \int_{R_f}L(\bY)d\QQ(\bY).
    \end{equation*}
    On the other hand, our assumption on $\alpha(f)$ is equivalent to
    \begin{equation*}
        \QQ[R_f] \leq \QQ[L(\bY) > \eta].
    \end{equation*}
    Thus, we are interested in solving the optimization
    \begin{equation*}
        \begin{array}{ll}
        \text{maximize} & \int_{R_f}L(\bY)d\QQ(\bY) \\
        \text{subject to} & R_f \in \sF, \\ & \QQ[R_f] \leq \QQ[L(\bY) > \eta].
        \end{array}
    \end{equation*}
    
    \noindent From this form it is intuitive how to proceed: let us write $R_\star \colonequals \{\bY: L(\bY) > \eta\} = R_{L_\eta}$, then the difference of powers is
    \begin{align*}
        (1 - \beta(L_\eta)) - (1 - \beta(f))
        &= \int_{R_\star}L(\bY)d\QQ(\bY) - \int_{R_f}L(\bY)d\QQ(\bY) \nonumber \\
        &= \int_{R_\star \setminus R_f}L(\bY)d\QQ(\bY) - \int_{R_f \setminus R_\star}L(\bY)d\QQ(\bY) \nonumber \\
        &\geq \eta\left(\QQ[R_\star \setminus R_f] - \QQ[R_f \setminus R_\star]\right) \nonumber \\
        &= \eta\left(\QQ[R_\star] - \QQ[R_f]\right) \nonumber \\
        &\geq 0,
    \end{align*}
    completing the proof.
\end{proof}

\subsection{Equivalence of Symmetric and Asymmetric Noise Models}
\label{app:symm}

For technical convenience, in the main text we worked with an asymmetric version of the spiked Wigner model (see Section~\ref{sec:spiked-matrix}), $\bY = \lambda \bx \bx^\top + \bZ$ where $\bZ$ has i.i.d.\ $\sN(0,1)$ entries. A more standard model is to instead observe $\widetilde\bY = \frac{1}{2}(\bY + \bY^\top) = \lambda \bx \bx^\top + \bW$, where $\bW$ is symmetric with $\sN(0,1)$ diagonal entries and $\sN(0,1/2)$ off-diagonal entries, all independent. These two models are equivalent, in the sense that if we are given a sample from one then we can produce a sample from the other. Clearly, if we are given $\bY$, we can symmetrize it to form $\widetilde\bY$. Conversely, if we are given $\widetilde\bY$, we can draw an independent matrix $\bG$ with i.i.d.\ $\sN(0,1)$ entries, and compute $\widetilde\bY + \frac{1}{2}(\bG - \bG^\top)$; one can check that the resulting matrix has the same distribution as $\bY$ (we are adding back the ``skew-symmetric part'' that is present in $\bY$ but not $\widetilde\bY$).

In the spiked tensor model (see Section~\ref{sec:spiked-tensor}), our asymmetric noise model is similarly equivalent to the standard symmetric model defined in \cite{RM-tensor} (in which the noise tensor $\bZ$ is averaged over all permutations of indices). Since we can treat each entry of the symmetric tensor separately, it is sufficient to show the following one-dimensional fact: for unknown $x \in \RR$, $k$ samples of the form $y_i = x + \sN(0,1)$ are equivalent to one sample of the form $\tilde y = x + \sN(0,1/k)$. Given $\{y_i\}$, we can sample $\tilde y$ by averaging: $\frac{1}{k}\sum_{i=1}^k y_i$. For the converse, fix unit vectors $\ba_1,\ldots,\ba_k$ at the corners of a simplex in $\RR^{k-1}$; these satisfy $\langle \ba_i,\ba_j \rangle = -\frac{1}{k-1}$ for all $i \ne j$. Given $\tilde y$, draw $\bu \sim \sN(0,{\bm I}_{k-1})$ and let $y_i = \tilde y + \sqrt{1-1/k} \,\langle \ba_i,\bu \rangle$; one can check that these have the correct distribution.

\subsection{Low-Degree Analysis of Spiked Wigner Above the PCA Threshold}
\label{app:wig-above}

\begin{proof}[Proof of Theorem~\ref{thm:wig-above}]
We follow the proof of Theorem~\ref{thm:tensor-lowdeg}(ii) in Section~\ref{sec:tensor-lower}. For any choice of $d \le D$, using the standard bound $\binom{2d}{d} \ge 4^d/(2\sqrt{d})$,
\begin{align*}
\|L_n^{\le D}\|^2 &\ge \frac{\lambda^{2d}}{d!} \Ex_{\bx^1,\bx^2}[\langle \bx^1,\bx^2 \rangle^{2d}] \\
&\ge \frac{\lambda^{2d}}{d!} \binom{n}{d} \frac{(2d)!}{2^{d}} \tag{using the moment bound~\eqref{eq:mom-bound} from Section~\ref{sec:tensor-lower}}\\
&= \frac{\lambda^{2d}}{d!} \frac{n!}{d!(n-d)!} \frac{(2d)!}{2^{d}} \\
&= \lambda^{2d} \binom{2d}{d} \frac{n!}{(n-d)! 2^d} \\
&\ge \lambda^{2d} \frac{4^d}{2\sqrt{d}} \frac{(n-d)^d}{2^d} \\
&= \frac{1}{2\sqrt{d}} \left(2\lambda^2 (n-d)\right)^d \\
&= \frac{1}{2\sqrt{d}} \left(\hat\lambda^2 \left(1 - \frac{d}{n}\right)\right)^d.
\end{align*}
Since $\hat\lambda > 1$, this diverges as $n \to \infty$ provided we choose $d \le D$ with $\omega(1) \leq d \leq o(n)$.
\end{proof}

\section{Hermite Polynomials}
\label{app:hermite}

Here we give definitions and basic facts regarding the Hermite polynomials (see, e.g, \cite{Szego-OP} for further details), which are orthogonal polynomials with respect to the standard Gaussian measure.

\begin{definition}
    \label{def:hermite}
    The \emph{univariate Hermite polynomials} are the sequence of polynomials $h_k(x) \in \RR[x]$ for $k \geq 0$ defined by the recursion
    \begin{align*}
        h_0(x) &= 1, \\
        h_{k + 1}(x) &= xh_k(x) - h_k^\prime(x).
    \end{align*}
    The \emph{normalized univariate Hermite polynomials} are $\what{h}_k(x) = h_k(x) / \sqrt{k!}$.
\end{definition}

\noindent
The following is the key property of the Hermite polynomials, which allows functions in $L^2(\sN(0, 1))$ to be expanded in terms of them.

\begin{proposition}
    The normalized univariate Hermite polynomials form a complete orthonormal system of polynomials for $L^2(\sN(0, 1))$.
\end{proposition}

The following are the multivariate generalizations of the above definition that we used throughout the main text.

\begin{definition}
    The \emph{$N$-variate Hermite polynomials} are the polynomials $H_{\bm\alpha}(\bX) \colonequals \prod_{i = 1}^N h_{\alpha_i}(X_i)$ for $\bm\alpha \in \NN^N$.
    The \emph{normalized $N$-variate Hermite polynomials in $N$ variables} are the polynomials $\what{H}_{\bm\alpha}(\bX) \colonequals \prod_{i = 1}^N \what{h}_{\alpha_i}(X_i) = (\prod_{i = 1}^N \alpha_i!)^{-1/2} \prod_{i = 1}^N h_{\alpha_i}(X_i)$ for $\bm\alpha \in \NN^N$.
\end{definition}

\noindent
Again, the following is the key property justifying expansions in terms of these polynomials.

\begin{proposition}
    The normalized $N$-variate Hermite polynomials form a complete orthonormal system of (multivariate) polynomials for $L^2(\sN(\bm 0, \bm I_N))$.
\end{proposition}

For the sake of completeness, we also provide proofs below of the three identities concerning univariate Hermite polynomials that we used in Section~\ref{sec:agn-low-deg} to derive the norm of the LDLR under the additive Gaussian noise model.
It is more convenient to prove these in a different order than they were presented in Section~\ref{sec:agn-low-deg}, since one identity is especially useful for proving the others.

\begin{proof}[Proof of Proposition~\ref{prop:gaussian-ibp} (Integration by Parts)]
    Recall that we are assuming a function $f: \RR \to \RR$ is $k$ times continuously differentiable and $f$ and its derivatives are $O(\exp(|x|^\alpha))$ for $\alpha \in (0, 2)$, and we want to show the identity
    \[ \Ex_{y \sim \sN(0, 1)}[h_k(y) f(y)] = \Ex_{y \sim \sN(0, 1)}\left[ \frac{d^k f}{dy^k}(y)\right]. \]
    We proceed by induction.
    Since $h_0(y) = 1$, the case $k = 0$ follows immediately.
    We also verify by hand the case $k = 1$, with $h_1(y) = y$:
    \[ \Ex_{y \sim \sN(0, 1)}\left[yf(y) \right] = \frac{1}{\sqrt{2\pi}} \int_{-\infty}^\infty f(y) \cdot ye^{-y^2 / 2}dy = \frac{1}{\sqrt{2\pi}} \int_{-\infty}^\infty f^\prime(y) e^{-y^2 / 2}dy = \Ex_{y \sim \sN(0, 1)}\left[f^\prime(y) \right], \]
    where we have used ordinary integration by parts.
    
    Now, suppose the identity holds for all degrees smaller than some $k \geq 2$, and expand the degree $k$ case according to the recursion:
    \begin{align*}
        \Ex_{y \sim \sN(0, 1)}[h_k(y) f(y)] 
        &= \Ex_{y \sim \sN(0, 1)}[y h_{k - 1}(y) f(y)] - \Ex_{y \sim \sN(0, 1)}[h_{k - 1}^\prime(y) f(y)] \\
        &= \Ex_{y \sim \sN(0, 1)}[h_{k - 1}^\prime(y)f(y)] + \Ex_{y \sim \sN(0, 1)}[h_{k - 1}(y)f^\prime(y)] - \Ex_{y \sim \sN(0, 1)}[h_{k - 1}^\prime(y) f(y)] \\
        &= \Ex_{y \sim \sN(0, 1)}[h_{k - 1}(y)f^\prime(y)] \\
        &= \Ex_{y \sim \sN(0, 1)}\left[\frac{d^k f}{dy^k}(y)\right],
    \end{align*}
    where we have used the degree 1 and then the degree $k - 1$ hypotheses.
\end{proof}

\begin{proof}[Proof of Proposition~\ref{prop:hermite-translation} (Translation Identity)]
    Recall that we want to show, for all $k \geq 0$ and $\mu \in \RR$, that
    \[ \Ex_{y \sim \sN(\mu, 1)}[h_k(y)] = \mu^k. \]
    We proceed by induction on $k$.
    Since $h_0(y) = 1$, the case $k = 0$ is immediate.
    Now, suppose the identity holds for degree $k - 1$, and expand the degree $k$ case according to the recursion:
    \begin{align*}
    \Ex_{y \sim \sN(\mu, 1)}[h_k(y)] 
    &= \Ex_{y \sim \sN(0, 1)}[h_k(\mu + y)] \\
    &= \mu \Ex_{y \sim \sN(0, 1)}[h_{k - 1}(\mu + y)] + \Ex_{y \sim \sN(0, 1)}[y h_{k - 1}(\mu + y)] - \Ex_{y \sim \sN(0, 1)}[h_{k - 1}^\prime(\mu + y)]
    \intertext{which may be simplified by the Gaussian integration by parts to}
    &= \mu \Ex_{y \sim \sN(0, 1)}[h_{k - 1}(\mu + y)] + \Ex_{y \sim \sN(0, 1)}[h_{k - 1}^\prime(\mu + y)] - \Ex_{y \sim \sN(0, 1)}[h_{k - 1}^\prime(\mu + y)] \\
    &= \mu \Ex_{y \sim \sN(0, 1)}[h_{k - 1}(\mu + y)],
    \end{align*}
    and the result follows by the inductive hypothesis.
\end{proof}

\begin{proof}[Proof of Proposition~\ref{prop:hermite-gf} (Generating Function)]
    Recall that we want to show the series identity for any $x, y \in \RR$,
    \[ \exp\left(xy - \frac{1}{2}x^2\right) = \sum_{k = 0}^\infty \frac{1}{k!}x^k h_k(y). \]
    For any fixed $x$, the left-hand side belongs to $L^2(\sN(0, 1))$ in the variable $y$.
    Thus this is merely a claim about the Hermite coefficients of this function, which may be computed by taking inner products.
    Namely, let us write
    \[ f_x(y) \colonequals \exp\left(xy - \frac{1}{2}x^2\right), \]
    then using Gaussian integration by parts,
    \[ \la f_x, \what{h}_k \ra = \frac{1}{\sqrt{k!}}\Ex_{y \sim \sN(0, 1)}\left[f_x(y) h_k(y)\right] = \frac{1}{\sqrt{k!}}\Ex_{y \sim \sN(0, 1)}\left[\frac{d^k f_x}{dy^k}(y) \right] = \frac{1}{\sqrt{k!}}x^k \Ex_{y \sim \sN(0, 1)}\left[f_x(y) \right]. \]
    A simple calculation shows that $\EE_{y \sim \sN(0, 1)}[f_x(y)] = 1$ (this is an evaluation of the Gaussian moment-generating function that we have mentioned in the main text), and then by the Hermite expansion
    \[ f_x(y) = \sum_{k = 0}^\infty \la f_x, \what{h}_k \ra \what{h}_k(y) = \sum_{k = 0}^\infty \frac{1}{k!}x^k h_k(y), \]
    giving the result.
\end{proof}

\section{Subgaussian Random Variables}
\label{app:subg}

Many of our rigorous arguments rely on the concept of \emph{subgaussianity}, which we now define. See, e.g., \cite{rig-notes} for more details.

\begin{definition}\label{def:subg}
For $\sigma^2 > 0$, we say that a real-valued random variable $\pi$ is $\sigma^2$-\emph{subgaussian} if $\EE[\pi] = 0$ and for all $t \in \RR$, the moment-generating function $M(t) = \EE[\exp(t \pi)]$ of $\pi$ exists and is bounded by $M(t) \le \exp(\sigma^2 t^2 / 2)$.
\end{definition}
\noindent Here $\sigma^2$ is called the \emph{variance proxy}, which is not necessarily equal to the variance of $\pi$ (although it can be shown that $\sigma^2 \ge \mathrm{Var}[\pi]$). The name \emph{subgaussian} refers to the fact that $\exp(\sigma^2 t^2 / 2)$ is the moment-generating function of $\sN(0,\sigma^2)$. 

The following are some examples of (laws of) subgaussian random variables. Clearly, $\sN(0,\sigma^2)$ is $\sigma^2$-subgaussian. By Hoeffding's lemma, any distribution supported on an interval $[a,b]$ is $(b-a)^2/4$-subgaussian. In particular, the Rademacher distribution $\Unif(\{\pm 1\})$ is $1$-subgaussian. Note also that the sum of $n$ independent $\sigma^2$-subgaussian random variables is $\sigma^2 n$-subgaussian.

Subgaussian random variables admit the following bound on their absolute moments; see Lemmas~1.3 and 1.4 of \cite{rig-notes}.
\begin{proposition}\label{prop:subg-mom}
If $\pi$ is $\sigma^2$-subgaussian then
$$\EE[|\pi|^k] \le (2\sigma^2)^{k/2} k \Gamma(k/2)$$
for every integer $k \ge 1$.
\end{proposition}

\noindent Here $\Gamma(\cdot)$ denotes the gamma function which, recall, is defined for all positive real numbers and satisfies $\Gamma(k) = (k-1)!$ when $k$ is a positive integer. We will need the following property of the gamma function.
\begin{proposition}\label{prop:gamma}
For all $x > 0$ and $a > 0$,
$$\frac{\Gamma(x+a)}{\Gamma(x)} \le (x+a)^a.$$
\end{proposition}
\begin{proof}
This follows from two standard properties of the gamma function. The first is that (similarly to the factorial) $\Gamma(x+1)/\Gamma(x) = x$ for all $x > 0$. The second is \emph{Gautschi's inequality}, which states that $\Gamma(x+s)/\Gamma(x) < (x+s)^{s}$ for all $x > 0$ and $s \in (0,1)$.
\end{proof}

In the context of the spiked Wigner model (Section~\ref{sec:spiked-matrix}), we now prove that subgaussian spike priors admit a local Chernoff bound (Definition~\ref{def:local-chernoff}).

\begin{proposition}\label{prop:local-chernoff}
Suppose $\pi$ is $\sigma^2$-subgaussian (for some constant $\sigma^2 > 0$) with $\EE[\pi] = 0$ and $\EE[\pi^2] = 1$. Let $(\sX_n)$ be the spike prior that draws each entry of $\bx$ i.i.d.\ from $\pi$ (where $\pi$ does not depend on $n$). Then $(\sX_n)$ admits a local Chernoff bound.
\end{proposition}

\begin{proof}
Since $\pi$ is subgaussian, $\pi^2$ is \emph{subexponential}, which implies $\EE[\exp(t \pi^2)] < \infty$ for all $|t| \le s$ for some $s > 0$ (see e.g., Lemma~1.12 of \cite{rig-notes}).

Let $\pi, \pi^\prime$ be independent copies of $\pi$, and set $\Pi = \pi \pi^\prime$. The moment-generating function of $\Pi$ is
\[ M(t) = \EE[\exp(t \Pi)] = \EE_\pi \EE_{\pi'}[\exp(t \pi \pi')] \le \EE_\pi\left[\exp\left(\sigma^2 t^2 \pi^2/2\right)\right] < \infty \]
provided $\frac{1}{2}\sigma^2 t^2 < s$, i.e.\ $|t| < \sqrt{2s/\sigma^2}$. Thus $M(t)$ exists in an open interval containing $t=0$, which implies $M'(0) = \EE[\Pi] = 0$ and $M''(0) = \EE[\Pi^2] = 1$ (this is the defining property of the moment-generating function: its derivatives at zero are the moments).

Let $\eta > 0$ and $f(t) \colonequals \exp\left(\frac{t^2}{2(1-\eta)}\right)$. Since $M(0) = 1, M'(0) = 0, M''(0) = 1$ and, as one may check, $f(0) = 1, f'(0) = 0, f''(0) = \frac{1}{1-\eta} > 1$, there exists $\delta > 0$ such that, for all $t \in [-\delta,\delta]$, $M(t)$ exists and $M(t) \le f(t)$.

We then apply the standard Chernoff bound argument to $\la \bx^1,\bx^2 \ra = \sum_{i=1}^n \Pi_i$ where $\Pi_1,\ldots,\Pi_n$ are i.i.d.\ copies of $\Pi$. For any $\alpha > 0$,
\begin{align*}
\Pr\left\{\la \bx^1,\bx^2 \ra \ge t\right\} &= \Pr\left\{\exp(\alpha \la \bx^1,\bx^2 \ra) \ge \exp(\alpha t)\right\}\\
&\le \exp(-\alpha t) \EE[\exp(\alpha \la \bx^1,\bx^2 \ra)] \tag{by Markov's inequality}\\
&= \exp(-\alpha t) \EE\left[\exp\left(\alpha \sum_{i=1}^n \Pi_i\right)\right]\\
&= \exp(-\alpha t) [M(\alpha)]^n\\
&\le \exp(-\alpha t) [f(\alpha)]^n \tag{provided $\alpha \le \delta$} \\
&= \exp(-\alpha t) \exp\left(\frac{\alpha^2 n}{2(1-\eta)}\right).
\end{align*}
Taking $\alpha = (1-\eta)t/n$,
$$\Pr\left\{\la \bx^1,\bx^2 \ra \ge t\right\} \le \exp\left(-\frac{1}{n}(1-\eta)t^2 + \frac{1}{2n}(1-\eta)t^2\right) = \exp\left(-\frac{1}{2n}(1-\eta)t^2\right)$$
as desired. This holds provided $\alpha \le \delta$, i.e.\ $t \le \delta n/(1-\eta)$. A symmetric argument with $-\Pi$ in place of $\Pi$ holds for the other tail, $\Pr\left\{\la \bx^1,\bx^2 \ra \le -t\right\}$.
\end{proof}

\section{Formal Consequences of the Low-Degree Method}

Here we provide the proofs of Theorems~\ref{thm:thresh-poly-hard} and \ref{thm:spectral-hard}, which show that if the LDLR predicts hardness then polynomial thresholding and low-degree spectral methods (respectively) must fail in a particular sense. We first discuss \emph{hypercontractivity}, the key ingredient in the proofs.

\subsection{Hypercontractivity}
\label{app:hyp}

The following hypercontractivity result states that the moments of low-degree polynomials of i.i.d.\ random variables must behave somewhat reasonably. The Rademacher version is the \emph{Bonami lemma} from \cite{boolean-book}, and the Gaussian version appears in \cite{janson-book} (see Theorem~5.10 and Remark~5.11 of \cite{janson-book}). We refer the reader to \cite{boolean-book} for a general discussion of hypercontractivity.

\begin{proposition}[Bonami Lemma]
\label{prop:bonami}
Let $\bx = (x_1,\ldots,x_n)$ have either i.i.d.\ $\sN(0,1)$ or i.i.d.\ Rademacher (uniform $\pm 1$) entries, and let $f: \RR^n \to \RR$ be a polynomial of degree $k$. Then
$$\EE[f(x)^4] \le 3^{2k} \,\EE[f(x)^2]^2.$$
\end{proposition}

\noindent We will combine this with the following standard second moment method.

\begin{proposition}[Paley-Zygmund Inequality]
\label{prop:P-Z}
If $Z \ge 0$ is a random variable with finite variance, and $0 \le \theta \le 1$, then
$$\mathrm{Pr}\left\{Z > \theta\, \EE[Z]\right\} \ge (1-\theta)^2 \frac{\EE[Z]^2}{\EE[Z^2]}.$$
\end{proposition}

\noindent By combining Propositions~\ref{prop:P-Z} and \ref{prop:bonami}, we immediately have the following.

\begin{corollary}\label{cor:hyp}
Let $\bx = (x_1,\ldots,x_n)$ have either i.i.d.\ $\sN(0,1)$ or i.i.d.\ Rademacher (uniform $\pm 1$) entries, and let $f: \RR^n \to \RR$ be a polynomial of degree $k$. Then, for $0 \le \theta \le 1$,
$$\Pr\left\{f(x)^2 > \theta\, \EE[f(x)^2]\right\} \ge (1-\theta)^2 \frac{\EE[f(x)^2]^2}{\EE[f(x)^4]} \ge \frac{(1-\theta)^2}{3^{2k}}.$$
\end{corollary}

\begin{remark}\label{rem:low-dom}
One rough interpretation of Corollary~\ref{cor:hyp} is that if $f$ is degree $k$, then $\EE[f(x)^2]$ cannot be dominated by an event of probability smaller than roughly $3^{-2k}$.
\end{remark}

\subsection{Lower Bound Against Thresholding Polynomials}
\label{app:thresh-poly-hard}

\noindent Using the above tools, we can now prove Theorem~\ref{thm:thresh-poly-hard}.

\begin{proof}[Proof of Theorem~\ref{thm:thresh-poly-hard}]
Consider the degree-$2kd$ polynomial $f^{2k}$. Using Jensen's inequality,
$$\EE_\PP[f(\bY)^{2k}] \ge \left(\EE_\PP[f(\bY)]\right)^{2k} \ge A^{2k}.$$
If $\EE_\QQ[f(\bY)^{4k}] \le B^{4k}$, then \eqref{eq:EQ-bound} holds immediately.
Otherwise, let $\theta \defeq B^{4k}/\EE_\QQ[f(\bY)^{4k}] \le 1$. Then, we have
\begin{align*}
\delta &\ge \QQ(|f(\bY)| \ge B)\\
&= \QQ(f(\bY)^{4k} \ge B^{4k})\\
&= \QQ(f(\bY)^{4k} \ge \theta\, \EE_\QQ[f(\bY)^{4k}])\\
&\ge \frac{(1-\theta)^2}{3^{4kd}}. \tag{by Proposition~\ref{cor:hyp}}
\end{align*}
Thus $\theta \ge 1 - 3^{2kd} \sqrt{\delta}$, implying
\begin{equation}\label{eq:EQ-bound}
\EE_\QQ[f(\bY)^{4k}] = \frac{B^{4k}}{\theta} \le B^{4k}\left(1 - 3^{2kd} \sqrt{\delta}\right)^{-1}.
\end{equation}
Using the key variational property of the LDLR (Proposition~\ref{prop:lr-optimal-l2}),
$$\|L^{\le 2kd}\| \ge \frac{\EE_\PP[f(\bY)^{2k}]}{\sqrt{\EE_\QQ[f(\bY)^{4k}]}} \ge \left(\frac{A}{B}\right)^{2k}\sqrt{1 - 3^{2kd} \sqrt{\delta}} \ge \frac{1}{2} \left(\frac{A}{B}\right)^{2k}$$
since $\delta \le \frac{1}{2} \cdot 3^{-4kd}$.
\end{proof}

\subsection{Lower Bound Against Spectral Methods}
\label{app:spectral-hard}

The proof of Theorem~\ref{thm:spectral-hard} is similar to the above.

\begin{proof}[Proof of Theorem~\ref{thm:spectral-hard}]
Let $\{\lambda_i\}$ be the eigenvalues of $\bM$, with $|\lambda_1| \ge \cdots \ge |\lambda_L|$. Consider the degree-$2kd$ polynomial $f(\bY) \defeq \Tr(\bM^{2k}) = \sum_{i=1}^L \lambda_i^{2k}$. Using Jensen's inequality,
$$\EE_\PP[f(\bY)] = \EE_\PP \sum_{i=1}^L \lambda_i^{2k} \ge \EE_\PP[\lambda_1^{2k}] \ge \left(\EE_\PP|\lambda_1|\right)^{2k} \ge A^{2k}.$$
If $\EE_\QQ[f(\bY)^2] \le L^2 B^{4k}$, then \eqref{eq:EQ-bound-2} holds immediately. 
Otherwise, let $\theta \defeq L^2 B^{4k}/\EE_\QQ[f(\bY)^2] \le 1$. Then, we have
\begin{align*}
\delta &\ge \QQ(\|\bM\| \ge B)\\
&= \QQ(\lambda_1^{2k} \ge B^{2k})\\
&\ge \QQ(f(\bY) \ge LB^{2k})\\
&= \QQ(f(\bY)^2 \ge L^2 B^{4k})\\
&= \QQ(f(\bY)^2 \ge \theta \,\EE_\QQ[f(\bY)^2])\\
&\ge \frac{(1-\theta)^2}{3^{4kd}}. \tag{by Proposition~\ref{cor:hyp}}
\end{align*}
Thus $\theta \ge 1 - 3^{2kd} \sqrt{\delta}$, implying
\begin{equation}\label{eq:EQ-bound-2}
\EE_\QQ[f(\bY)^2] = \frac{L^2 B^{4k}}{\theta} \le L^2 B^{4k}\left(1 - 3^{2kd} \sqrt{\delta}\right)^{-1}.
\end{equation}
Using the key variational property of the LDLR (Proposition~\ref{prop:lr-optimal-l2}),
$$\|L^{\le 2kd}\| \ge \frac{\EE_\PP[f(\bY)]}{\sqrt{\EE_\QQ[f(\bY)^2]}} \ge \frac{1}{L}\left(\frac{A}{B}\right)^{2k}\sqrt{1 - 3^{2kd} \sqrt{\delta}} \ge \frac{1}{2L}\left(\frac{A}{B}\right)^{2k}$$
since $\delta \le \frac{1}{2} \cdot 3^{-4kd}$.
\end{proof}

\end{document}